\documentclass[11pt]{article}

\usepackage{amsthm}
\usepackage{amsfonts}
\usepackage{amsmath}
\usepackage{amssymb}
\usepackage{color}
\usepackage{easybmat}
\usepackage{blkarray}
\usepackage{datetime}
\usepackage{hyperref}
\usepackage{comment}
\usepackage{accents}

\usepackage[T2A]{fontenc}

\usepackage{mathrsfs}
\renewcommand{\mathcal}[1]{\mathscr{#1}}

\usepackage{marvosym}

\usepackage{theory}
\usepackage{theorems}

\newcommand{\ta}{\oplus}
\newcommand{\tp}{\odot}

\renewcommand{\phi}{\varphi}

\newcommand{\Sing}{\mathsf{Sing}}

\newcommand{\bb}[1]{\mathbb{#1}}

\newcommand{\p}{\prime}

\newcommand\restr[2]{{% we make the whole thing an ordinary symbol
  \left.\kern-\nulldelimiterspace % automatically resize the bar with \right
  #1 % the function
  \vphantom{\big|} % pretend it's a little taller at normal size
  \right|_{#2} % this is the delimiter
}}

\newcommand{\Ml}{\mathit{Ml}}
\newcommand{\Mr}{\mathit{Mr}}
\newcommand{\Min}{\mathit{Min}}
\newcommand{\Max}{\mathit{Max}}

\ddmmyyyydate
\title{Tropical Effective \\ Primary and Dual Nullstellens\"atze\footnote{An extended abstract of a preliminary version~\cite{STACS} appeared in the proceedings of the 32nd International Symposium on Theoretical Aspects of Computer Science (STACS 2015).}}

\author{Dima Grigoriev$^1$, Vladimir V. Podolskii$^{2,3}$\\[3pt]
$^1$\small CNRS, Math\'ematiques, Universit\'e de Lille, France\\
\small \href{mailto:Dmitry.Grigoryev@math.univ-lille1.fr}{Dmitry.Grigoryev@math.univ-lille1.fr}\\
$^2$ \small Steklov Mathematical Institute, Moscow, Russia\\
$^3$ \small National Research University Higher School of Economics, Moscow, Russia\\
 \small   \href{mailto:podolskii@mi.ras.ru}{podolskii@mi.ras.ru}
}

\date{}
\begin{document}

\maketitle

\begin{abstract}
Tropical algebra is an emerging field with a number of applications in various areas of mathematics.
In many of these applications appeal to tropical polynomials allows to study properties
of mathematical objects such as algebraic varieties and algebraic curves from the computational point of view.
This makes it important to study both mathematical and computational aspects of tropical polynomials.

In this paper we prove a tropical Nullstellensatz and moreover we show an effective formulation of this theorem.
Nullstellensatz is a natural step in building algebraic theory of tropical polynomials and its
effective version is relevant for computational aspects of this field.

On our way we establish a simple formulation of min-plus and tropical linear dualities.
We also observe a close connection between tropical and min-plus polynomial systems.

%This paper is in the line of research trying to build foundations of algebraic theory
%of tropical (or min-plus) algebra, a field studying tropical semi-ring of real
%numbers with operations $+$ and $\min$.

%In this paper we prove a tropical analog of classical Nullstellensatz.
%We show different versions of the theorem in different settings. More specifically we show
%\begin{enumerate}
%\item both primary and dual versions of the theorem;
%\item both versions with the presence of $\infty$ in the semi-ring
%and without one;
%\item analogous theorems for the cases of tropical polynomials and min-plus polynomials.
%\end{enumerate}
%All our theorems are effective and we show that our bounds are tight.
%We show an interesting peculiarity that the level of effectiveness of Nullstellensatz
%depends substantially on the presence of $\infty$ in the semiring.

%On our way we establish a simple formulation of min-plus and tropical linear dualities.
%We also observe a close connection between tropical and min-plus polynomial systems.

\end{abstract}

\tableofcontents

\section{Introduction}

A \emph{min-plus} or a \emph{tropical semiring} is defined by a set $\bb{K}$, which can be $\mathbb{R}$, $ \bb{R}_\infty = \mathbb{R} \cup \{+\infty\}$,
$\mathbb{Q}$ or $\bb{Q}_\infty = \mathbb{Q} \cup \{+\infty\}$
endowed with two operations, \emph{tropical addition} $\ta$ and \emph{tropical multiplication} $\tp$, defined in
the following way:
$$
x \ta y = \min\{x,y\}, \ \ \ \ x \tp y = x + y.
$$

Tropical polynomials are a natural analog of classical polynomials.
In classical terms tropical polynomial is an expression of the form $f(\vec{x}) = \min_i M_i(\vec{x})$,
where each $M_i(\vec{x})$ is a linear polynomial (a tropical monomial) in variables $\vec{x} = (x_1, \ldots, x_n)$,
and all coefficients of all $M_i$'s are nonnegative integers except for a free coefficient which can be any element of $\bb{K}$.

The degree of a tropical monomial $M$ is the sum of its coefficients (except the free coefficient) and
the degree of a tropical polynomial $f$ denoted by $\deg(f)$ is the maximal degree of its monomials.
A point $\vec{a} \in \bb{K}^n$ is a root of the polynomial $f$ if the minimum $\min_i\{M_i(\vec{a})\}$
is either attained on at least two different monomials $M_i$, or is infinite.
We defer more detailed definitions on the basics of min-plus algebra to Preliminaries.

Tropical polynomials have appeared in various areas of mathematics and found many applications (see, for example,~\cite{IMS2009tropical,MS2015tropical,
sturmfels02equations,Mikhalkin2004survey,
RGST05first_steps,HuberS95polyhedral};
one of the earliest papers in tropical mathematics is~\cite{Vorobyev67}).
An important advantage of tropical algebra is that it makes some properties of classical mathematical objects computationally accessible~\cite{theobald06frontiers,IMS2009tropical,MS2015tropical,sturmfels02equations}.
One of the main goals of min-plus mathematics is to build a theory of tropical polynomials which would help to work with
them and would possibly lead to new results in the related areas. Computational reasons, on the other hand, make
it important to keep the theory maximally computationally efficient.

The best studied so far is the case of tropical linear polynomials and systems of tropical linear polynomials.
For them an analog of the large part of the classical theory of linear polynomials was established.
This includes studies of tropical analogs of the rank of a matrix and the independence of vectors~\cite{DSS05rank,IzhakianR2009rank,AkianGG09rank},
an analog of the determinant of a matrix and its properties~\cite{RGST05first_steps}, an analog of Gauss triangular form~\cite{Grigoriev13complexity}. Also the solvability problem for tropical linear systems was studied from the complexity point of view.
Interestingly, this problem turns out to be polynomially equivalent to the mean payoff games problem~\cite{GP13Complexity} which received considerable attention in computational complexity theory.

For tropical polynomials of arbitrary degree less is known. In~\cite{IzhakianS07} the radical of a tropical
ideal was explicitly described.
In~\cite{RGST05first_steps,ST10SIAM} a tropical version of the Bezout theorem was proved for tropical polynomial systems for the case when the number of polynomials in the system is equal to the number of variables.
In~\cite{theobald06frontiers} it was shown that the solvability problem for tropical polynomial systems is $\NP$-complete.

Along with tropical polynomials there were also studied min-plus polynomials.
Min-plus polynomial is a pair of tropical polynomials $(f(\vec{x}),g(\vec{x}))$.
A point $\vec{a} \in \bb{K}^n$ is a root of the polynomial $(f(\vec{x}),g(\vec{x}))$ if $f(\vec{a}) = g(\vec{a})$.
We call an equation $f(\vec{x}) = g(\vec{x})$ \emph{min-plus polynomial equation}.

Min-plus polynomials were studied mainly for their connections to dynamic programming (see~\cite{butkovic10systems,jukna14circuits}).
As in the case of tropical polynomials here the best studied case is the case of linear min-plus polynomials~\cite{butkovic10systems}.
Also in~\cite{GP13Complexity} the connection between min-plus and tropical linear polynomials was established.

As for the min-plus polynomials of arbitrary degree much less is known. We are only aware of the result on the computational complexity
of the solvability problem of a system of min-plus polynomials: the paper~\cite{GrigorievS14cryptography} shows that this problem is $\NP$-complete.

\paragraph{Our results.}

A next natural step in developing of the theory of tropical polynomials would be an analog of the classical Hilbert's Nullstellensatz,
which for the classical polynomials constitutes one of the cornerstones of algebraic geometry.
Concerning the tropical Nullstellensatz, the problem was already addressed in the paper~\cite{Grigoriev12null}.
That paper came up with a general idea to approach this theorem in the tropical case
through the dual formulation (a naive tropical analog of Nullstellensatz trivially fails, see below). Moreover, in~\cite{Grigoriev12null} there was formulated a conjecture
(which we restate below as Conjecture~\ref{con:nullstellensatz}) capturing the formulation of the tropical dual Nullstellensatz
and this conjecture was proven for the case of polynomials in $1$ variable.
Previously  in~\cite{tabera2008} a tropical dual Nullstellensatz was established for a pair of polynomials
in 1 variable. This result relied on the classical resultant and on the Kapranov's theorem~\cite{EinsiedlerKL07,tabera2008}.

More specifically, in~\cite{Grigoriev12null} there was considered the Macaulay matrix of a system of tropical polynomials $F = \{f_1, \ldots, f_k\}$.
This matrix can be easily constructed from $F$: we just consider all the polynomials of the form $f_i + m_j$ (in classical notation) of degree at most $N$, where $N$ is a parameter and $m_j$ is a tropical monomial.
We put the coefficients of these polynomials in the rows of the matrix, while the columns of the matrix correspond to monomials.
Empty entries of the matrix we fill with $\infty$. The resulting matrix we denote by $M_N$.
In~\cite{Grigoriev12null} it was conjectured that the system of polynomials $F$ has a root iff the tropical linear system with the matrix $M_N$ has a solution, and moreover $N$ can be bounded by some function on $n$, $k$ and the degree of polynomials in $F$ (this refers to effectiveness).

In this paper we prove this conjecture establishing an effective version of the tropical dual Nullstellensatz. 
%We also provide examples showing that our bounds on $N$ are close to tight. 
Surprisingly, it turns out that the cases of tropical semiring with and without $\infty$ differ dramatically.
More specifically, in the case of tropical semirings $\bb{K} = \bb{R}$ or $\bb{K} = \bb{Q}$ we show that $F$ has a root iff the tropical linear system with the
matrix $M_N$ has a solution, where $N = (n+2) k d$, $d$ is the maximal degree of polynomials in $F$, $k$ is the number of polynomials in $F$ and $n$ is the number of variables.
For the case of tropical semirings $\bb{K} = \bb{R}_{\infty}$ or $\bb{K} = \bb{Q}_{\infty}$ we show a similar result, but with $N = (C_1d)^{\min(n,k)+C_2}$ for some constants $C_1$ and $C_2$. Thus for the case without $\infty$ the bound on $N$ is polynomial in $n,k,d$ and for the case with $\infty$ the bound on $N$ is still polynomial in $d$, but is exponential in $n$ and $k$. We give examples showing that our bounds on $N$ are qualitatively optimal, that is the difference of the values of
$N$ in these cases is not an artifact of the proof, but is unavoidable. However, quantitatively there is a gap between upper and lower bounds,
see Section~\ref{sec:results} for details.

Regarding the substantial gap between the required degree in the finite and infinite cases we observe that there is a somewhat similar situation for the classical Nullstellensatz. Indeed, we show that in case of the semiring $\bb{R}$ the bound in a tropical effective Nullstellensatz is roughly equal to the sum of the degrees of the polynomials, while in case of the semiring $\bb{R}_{\infty}$ the bound is roughly equal to the product of the degrees (Theorems~\ref{thm:main} and~\ref{thm.tropical-primary}). We recall that for systems of classical polynomials over an algebraically closed field the bound in the effective Nullstellensatz is roughly equal to the sum of the degrees of polynomials in homogeneous (projective) case~\cite{Lazard77,Lazard81} while the bound is roughly equal to the product of the degrees for arbitrary polynomials (affine case)~\cite{Brownawell87,GuistiHS93,Kollar88}.

As a consequence of the tropical dual Nullstellensatz we obtain its infinite version. Namely, a system of tropical polynomials has a root iff the infinite tropical linear system with the infinite Macaulay matrix $M$ (that is, with no bound on the degree) has a solution. Note that the latter system is well defined since each row of $M$ contains just a finite number of finite entries. This infinite version was conjectured in~\cite{Grigoriev12null}, where it was also observed that a similar infinite dual version of the classical Nullstellensatz holds.

Next we show a primary version of the tropical Nullstellensatz.
We view Nullstellensatz as a duality\footnote{To avoid a confusion we note that we use the word `dual' in two different meanings. First, we use it in the term ``dual Nullstellensatz'' as opposed to the standard version of Nullstellensatz. This means that the dual Nullstellensatz is obtained from the standard Nullstellensatz by the (linear) duality. Second, we use the word `dual' in term ``duality result'' to denote the general type of results. Since the standard Nullstellensatz is a duality result itself, applying the linear duality to it results in a non-duality result. Thus, the dual Nullstellensatz is not a duality result in a proper sense, but rather the word "dual" is used in contrast to the customary Nullstellensatz which we name "primary".} result for systems of polynomials:
if there is \emph{no} root to the system of polynomials then some positive property holds (something \emph{does} exist).
In the classical case this positive property is the containment of $1$ in the ideal generated by polynomials (over an algebraically closed field).
A naive analog does not hold for the tropical case. 
Indeed, for example, the system of tropical polynomials $\{\min(x,0), \min(x,1)\}$ has no roots but the tropical ideal generated by this system 
does not contain a constant polynomial and more generally any polynomial in this tropical ideal has a root.
Basically, the point is that in the tropical semiring there is no subtraction,
so in any algebraic combination of tropical polynomials no monomials cancel out.
To overcome this difficulty we introduce the notion of a \emph{nonsingular} tropical algebraic combination of tropical
polynomials (see the definition in Preliminaries; here we only note that the nonsingularity property is simple and straightforward to check).
For the tropical primary Nullstellensatz we show that there is no root to the tropical polynomial system $F$ iff there is a nonsingular tropical algebraic combination of polynomials in $F$ of degree at most $N$. We show this result for both cases of tropical semiring with and without $\infty$
and the value of $N$ in both cases corresponds the value of $N$ in the tropical dual Nullstellensatz.

To establish the primary Nullstellensatz we need a duality for tropical linear systems.
We show this duality result as a sidestep.
However we note that the duality for tropical linear systems is heavily based on already known results~\cite{akian12mean_payoff} and should be considered more as an observation.

We also prove analogs of all mentioned results for the case of min-plus polynomials.
As another sidestep of our analysis we study the connection between tropical and min-plus systems of polynomials.
We argue that these two settings are very closely connected and that this connection can be used to establish new results in tropical algebra.
The observation is that some results (like linear duality) are easier to obtain for min-plus polynomials and then translate to tropical polynomials,
and some other results (like the dual Nullstellensatz), on the other hand, are easier to obtain for tropical polynomials and then translate to min-plus polynomials.
In our opinion it is fruitful for further development of the theory to consider both settings simultaneously.

\paragraph{Our techniques}

We use the general approach of the paper~\cite{Grigoriev12null} to Nullstellensatz through the dual formulation.

To establish the dual Nullstellensatz we use methods of discrete geometry dealing with integer polyhedra.
First we obtain dual Nullstellensatz for the case without $\infty$.
The case with $\infty$ requires much more additional technical work.

To obtain the primary Nullstellensatz we apply the duality results for tropical linear polynomials.
We note that these results rely on the completely different combinatorial techniques, namely on the connection
to mean payoff games~\cite{akian12mean_payoff}.

\paragraph{Other works on tropical Nullstellensatz} In~\cite{Izhakian08} there was established Nullstellensatz
for the tropical semiring augmented with additional elements (called ghosts).
This result is in the line with other results~\cite{sturmfels02equations} trying to capture tropical mathematics by the means
of the classical ones. However, the tropical semiring augmented with ghosts constitutes (logically) a completely
different model compared to the usual tropical semiring.
Thus our results are incomparable with the ones of~\cite{Izhakian08}.

We also note that~\cite{IzhakianS07} (which has Nullstellensatz in the title) takes completely different view
on Nullstellensatz. We consider Nullstellensatz as a result on the solvability of a system of polynomials,
and~\cite{IzhakianS07} views Nullstellensatz as a result on the structure of the radical of a tropical
ideal. As it can be easily seen, for example, from our results during the translation from the classical world to the tropical one,
the connection between the solvability and the ideal changes drastically (cf. the example $F = \{\min(x,0), \min(x,1)\}$ above). Thus our results are incomparable with the results of~\cite{IzhakianS07} as well.

In~\cite{purbhoo2008} a version of Nullstellensatz was shown for a related structure of amoebas.
However,~\cite{purbhoo2008} proposes a view on Nullstellensatz different from the one suggested in the present paper. An analogous to~\cite{purbhoo2008} result in tropical setting was obtained in~\cite{Bogart2007}.

The rest of the paper is organized as follows. 
In Section~\ref{sec:preliminaries} we introduce main definitions. 
In Section~\ref{sec:results} we state our results. 
In Section~\ref{sec:tropical_dual} we prove the tropical and min-plus dual Nullstellensatz.
In Section~\ref{sec:trop_vs_min-plus} we establish a connection between the sets of roots of tropical and of min-plus polynomial systems.
As an illustration we deduce the min-plus dual Nullstellensatz  with slightly worse parameter from the tropical dual Nullstellensatzs.
In Section~\ref{sec:primary_nullstellensatz} we show the tropical and min-plus primary Nullstellens\"atze.
In Section~\ref{sec:linear_duality} we show the min-plus and tropical linear dualities.
Sections~\ref{sec:trop_vs_min-plus} and~\ref{sec:linear_duality} depend only on Sections~\ref{sec:preliminaries} and~\ref{sec:results}, and can be read independently of other sections.

\section{Preliminaries}
\label{sec:preliminaries}

\paragraph{Tropical and min-plus polynomials.}

A \emph{min-plus} or a \emph{tropical semiring} is defined by a set $\bb{K}$, which can be $\mathbb{R}$, $ \bb{R}_\infty = \mathbb{R} \cup \{+\infty\}$,
$\mathbb{Q}$ or $\bb{Q}_\infty = \mathbb{Q} \cup \{+\infty\}$
endowed with two operations, \emph{tropical addition} $\ta$ and \emph{tropical multiplication} $\tp$, defined in
the following way:
$$
x \ta y = \min\{x,y\}, \ \ \ \ x \tp y = x + y.
$$
Below we mainly consider $\bb{K}=\bb{R}$ and $\bb{K} = \bb{R}_{\infty}$.
The proofs however literally translate to the cases of $\bb{Q}$ and $\bb{Q}_\infty$.

A tropical (or min-plus) monomial in variables $\vec{x} = (x_1, \ldots, x_n)$ is defined as
\begin{equation} \label{eq:monomial}
m(\vec{x}) = c \tp x_1^{\tp i_1} \tp \ldots \tp x_n^{\tp i_n},
\end{equation}
where $c$ is an element of the semiring $\bb{K}$ and $i_1, \ldots, i_n$ are nonnegative integers.
In usual notation the monomial is a linear function
$$
m(\vec{x}) = c + i_1 x_1 + \ldots + i_n x_n.
$$
We denote $\vec{x} = (x_1, \ldots, x_n)$ and for $I = (i_1, \ldots, i_n)$
we introduce the notation
$$
\vec{x}^I = x_1^{\tp i_1} \tp \ldots \tp x_n^{\tp i_n} = i_1 x_1 + \ldots + i_n x_n.
$$
The degree of the monomial $m$ is defined as the sum $i_1 + \ldots + i_n$. We denote this sum by $|I|$.

A tropical polynomial is the tropical sum of tropical monomials
$$
f(\vec{x}) = \bigoplus_i m_i(\vec{x})
$$
with pairwise distinct exponent vectors $I = (i_1,\dots,i_n)$, or in the usual notation $f(\vec{x}) = \min_i m_i(\vec{x})$.
The degree of the tropical polynomial $f$ denoted by $\deg(f)$ is the maximal degree of its monomials.
A point $\vec{a} \in \bb{K}^n$ is a root of the polynomial $f$ if the minimum $\min_i\{m_i(\vec{a})\}$
is either attained on at least two different monomials $m_i$ or is infinite.

Geometrically, over the semiring $\bb{R}$ a tropical polynomial $f(\vec{x})$ is a convex piece-wise linear function over $\bb{R}^{n}$ and the roots of $f$ are non-smoothness points of this function.

We say that $\vec{a}$ is a root for the system of tropical polynomials $F = \{f_1, \ldots, f_k\}$ in variables $\vec{x}$ if $\vec{a}$ is a root to each polynomial $f_i \in F$.

A min-plus polynomial is a pair of tropical polynomials
$$
\left(f(\vec{x}), g(\vec{x})\right).
$$
The degree of a min-plus polynomial is the maximum of degrees of $f$ and $g$.
A point $\vec{a} \in \bb{K}^n$ is a root of this polynomial if the following equality holds: $f(\vec{a}) = g(\vec{a})$.
We call an equation $f(\vec{x}) = g(\vec{x})$ \emph{min-plus polynomial equation}.

\paragraph{Linear polynomials.} An important special case of tropical and min-plus polynomials are linear polynomials
which are just general tropical and min-plus polynomials of degree $1$.
If there is no monomial of degree $0$ (constant monomial) in a linear polynomial, we say that the linear polynomial is \emph{homogeneous}.

It is convenient to express a linear polynomial $f$ in the form
$$
\min_{1 \leq j \leq n} \{a_{j} + x_{j}\}.
$$
In particular, if some variable $x_j$ is not presented in the polynomial for notational convenience we still write it in this expression and just set the corresponding coefficient $a_j$ to $\infty$ (even if we consider the semiring $\bb{R}$). 
%Also we assume that no variables are presented in a linear polynomial twice. This does not lose generality since if there are several degree-1 monomials with the same variable, then the one with the smaller coefficient is always smaller than all others, so we can keep this particular monomial and discard all others without changing the solutions of the polynomial.

The tropical homogeneous linear system
\begin{equation} \label{eq.tropical}
\min_{1 \leq j \leq n} \{a_{ij} + x_{j}\},\ 1 \leq i \leq m,
\end{equation}
then can be naturally associated with its matrix $A \in \bb{R}^{m \times n}_{\infty}$.
We will also use a matrix notation $A \tp \vec{x}$ for such systems.
Thus, we consider tropical linear systems $A \tp \vec{x}$ with the matrices $A \in \bb{R}^{m\times n}_{\infty}$ over both semirings $\bb{R}$ and $\bb{R}_{\infty}$. We assume however that there are no rows and columns in $A$ consisting entirely of $\infty$. 
One can delete each infinite row since any vector is its root. One can also delete each infinite column since this has no effect on the solvability of the system. 
It is convenient to call the roots of tropical linear systems \emph{solutions}.

We note that it is also common to consider tropical linear systems $A \tp \vec{x}$ over $\bb{R}$ with matrices $A \in \bb{R}^{m\times n}$ only. Some of our sidestep results (Corollaries~\ref{cor:min-plus_linear_duality} and~\ref{cor:tropical_linear_duality}) address this setting.

Analogously min-plus homogeneous linear systems
\begin{equation*}
\left(\min_{1 \leq j \leq n} \{a_{ij} + x_{j}\},\  \min_{1 \leq j \leq n} \{b_{ij} + x_{j}\}\right),\ 1 \leq i \leq m,
\end{equation*}
can be associated with a pair of matrices $A$ and $B$ corresponding to the left-hand side
and the right-hand side of an equation. We will also write min-plus homogeneous linear system of equations in a matrix form as $A \tp \vec{x} = B \tp \vec{x}$.
It will be also convenient to consider min-plus linear systems of (componentwise) inequalities $A \tp \vec{x} \leq B \tp \vec{x}$.
It is not hard to see that their expressive power is the same as of equations.
\begin{lemma}
For any min-plus system of linear homogeneous equations there is an equivalent system of min-plus linear inequalities and visa versa.
\end{lemma}
\begin{proof}
Indeed, each min-plus linear equation $L_1(\vec{x}) = L_{2}(\vec{x})$ is equivalent to the pair of min-plus inequalities $L_1(\vec{x}) \geq L_2(\vec{x})$ and $L_1(\vec{x}) \leq L_2(\vec{x})$.
On the other hand min-plus linear inequality $L_1(\vec{x}) \leq L_{2}(\vec{x})$ is equivalent to the min-plus equation $L_{1}(\vec{x}) = \min(L_1(\vec{x}),L_2(\vec{x}))$.
It is not hard to see that the last equation can be transformed to the form of min-plus linear equation.
\end{proof}

There is one more important convention we make concerning the case of a tropical semiring with infinity.
For two matrices $A,B \in \bb{R}_{\infty}^{m\times n}$
we say that the system $A \tp \vec{x} < B \tp \vec{x}$ has a solution if
there is $\vec{x} \in \bb{R}_{\infty}^n$ such that for each row of the system if
one of sides is finite, then the strict inequality holds, but also the case where both sides are equal to $\infty$
is allowed (informally, we can say that $\infty < \infty$).

We also consider non-homogeneous tropical linear systems
\begin{equation}
\min_{1 \leq j \leq n} \{a_{ij} + x_{j}\} \cup \{a_i\},\ 1 \leq i \leq m.
\end{equation}
This system can be naturally associated to the matrix $A \in \bb{K}^{m \times (n+1)}$ and written in the matrix form as
$A \tp (\vec{x},0)$. Analogously, we can consider a non-homogeneous min-plus linear systems $A \tp (\vec{x},0) = B \tp (\vec{x},0)$.
We note that over $\bb{R}^n$ the tropical system $A \tp (\vec{x},0)$ is solvable iff the homogeneous system $A \tp \vec{x}^\prime$ is solvable,
where $\vec{x}^\prime = (\vec{x}, x_{n+1})$. Indeed, having a solution $\vec{x}^\prime$ for the latter system we can add the same number to all coordinates of $\vec{x}^\prime$ to make $x_{n+1} = 0$ and thus obtain a solution of the former system.
The same is true for the min-plus case. But this is not true over $\bb{R}_{\infty}$: a homogeneous system always has a solution (just let $\vec{x}= (\infty, \ldots, \infty)$), but a non-homogeneous system does not always have a solution. However, over $\bb{R}_{\infty}$ we have that $A \tp (\vec{x},0)$ has a solution iff therre is a solution to $A \tp \vec{x}^\prime$ with $x_{n+1}\neq \infty$.

\section{Results Statements}
\label{sec:results}

\subsection{Tropical and Min-plus Nullstellensatz}

\begin{definition} For a given system of tropical polynomials $F = \{f_1, \ldots, f_k\}$ in $n$ variables we introduce its infinite Macaulay matrix $M$.
The columns of $M$ correspond to nonnegative integer vectors $I \in \bb{Z}^n_{+}$ and the rows of $M$ correspond to the pairs $(j,J)$, where $1 \leq j \leq k$ and $J \in \bb{Z}^n_{+}$.
For a given $I$ and $(j,J)$ we let the entry $m_{(j,J), I}$ be equal to the coefficient of the monomial $\vec{x}^I$ in the polynomial
$\vec{x}^J \tp f_{j}$ (if there is no such monomial in the polynomial we assume that the entry is equal to $\infty$).
By $M_N$ we denote the finite submatrix of the matrix $M$ consisting of the columns $I$ such that $|I|=i_1 + \ldots + i_n \leq N$
and the rows that have all their finite entries in these columns.
The tropical linear system $M_N\tp \vec{y}$ will be of interest to us.
Over $\bb{R}_{\infty}$ we consider the non-homogeneous system $M_{N} \tp (0,\vec{y})$. The column corresponding
to the constant monomial is a non-homogeneous column.

For a system of min-plus polynomials $F = \{(f_1,g_1), \ldots, (f_k,g_k)\}$ we analogously introduce the pair of matrices $\Ml$ and $\Mr$
corresponding to the left-hand sides and the right-hand sides of polynomials respectively. In the same way we introduce matrices $\Ml_N$, $\Mr_N$
and the corresponding min-plus linear system $\Ml_N \tp \vec{y} = \Mr_N \tp \vec{y}$. Analogously, for the case of $\bb{R}_\infty$ we consider the non-homogeneous system $\Ml_N \tp (0,\vec{y}) = \Mr_N \tp (0,\vec{y})$.
\end{definition}

In~\cite{Grigoriev12null} there were conjectured three forms of a tropical dual Nullstellensatz.
We state the most strong of them, namely an effective Nullstellensatz conjecture.
\begin{conjecture}[\cite{Grigoriev12null}] \label{con:nullstellensatz}
There is a function
$N$ of $n$ and of $\deg(f_i)$ for $1 \leq i \leq k$ such that a system of polynomials $F = \{f_1, \ldots, f_k\}$
has a common tropical root iff the tropical linear system corresponding
to the matrix $M_N$ has a solution.
\end{conjecture}
Note that the classical analog of this statement is precisely the effective Nullstellensatz in the dual form
(see~\cite{Grigoriev12null} for the detailed discussion).

In~\cite{Grigoriev12null} the conjecture was proven for the case of $n=1$.
In this paper we prove the general case of the conjecture.

\begin{theorem}[Tropical Dual Nulstellensatz] \label{thm:main}
Consider a system of tropical polynomials $F=\{f_1, \ldots, f_k\}$ in $n$ variables. Denote by $d_i$ the degree of the polynomial $f_i$ and let $d = \max_i d_i$.
\renewcommand*{\theenumi}{\thetheorem(\roman{enumi})}%
  \renewcommand*{\labelenumi}{(\roman{enumi})}%
\begin{enumerate}

\item \label{thm:main1} Over the semiring $\bb{R}$ the system $F$ has a root iff the Macaulay tropical linear system $M_N \tp \vec{y}$ for
$$
N = (n+2)\left(d_1 + \ldots + d_k\right)
$$
has a solution.

\item \label{thm:main2} Over the semiring $\bb{R}_{\infty}$ the system $F$ has a root iff the Macaulay tropical non-homogeneous linear system $M_N \tp (0,\vec{y})$ for
$$
N = \poly(n,k,d) \left(4d\right)^{\min(n,k)}
$$
has a solution.

\end{enumerate}
\end{theorem}

Note that one direction of the theorem is simple. Indeed, if the system of polynomials has a root $\vec{a} \in \bb{R}^n$ then it is not hard to see that there is a solution to $M_N \tp \vec{y}$. Indeed, note that the coordinates $y_I$ of $\vec{y}$ correspond to monomials $\vec{x}^I$, so let $y_I = \vec{a}^I$. Since each row of $M_N$ correspond to the polynomial of the form $\vec{x}^J \tp f_j(\vec{x})$ and $\vec{a}$ is a root of any such polynomial, we have that $\vec{y}$ satisfies all rows of $M_N$. Thus the essence of the theorem is to prove the other direction. The same argument works over $\bb{R}_{\infty}$

We note that we can also consider an infinite Macaulay tropical linear system $M \tp \vec{y}$.
It is well defined since each row of $M$ has only finite number of finite entries.
As a corollary of the previous theorem we deduce an infinite version of the tropical dual Nullstellensatz.

\begin{corollary} \label{cor:tropical_infinite}
A system of tropical polynomials $F=\{f_1, \ldots, f_k\}$ of $n$ variables has a root iff the infinite Macaulay tropical linear system with the matrix $M$ has a solution.
The result holds for both $\bb{R}$ and $\bb{R}_\infty$ semirings.
\end{corollary}

The proof in the simple direction is the same as for Theorem~\ref{thm:main} and the hard part of the corollary follows trivially from Theorem~\ref{thm:main}.

We show a dual Nullstellensatz for the min-plus case.

\begin{theorem}[Min-Plus Dual Nullstellensatz] \label{thm.min-plus}
Consider a system of min-plus polynomials $F = \{(f_1,g_1), \ldots, (f_k,g_k)\}$ in $n$ variables. Let $d_i$ be the degree of the polynomial $(f_i, g_i)$ and let $d = \max_i d_i$.
\renewcommand*{\theenumi}{\thetheorem(\roman{enumi})}%
  \renewcommand*{\labelenumi}{(\roman{enumi})}%
\begin{enumerate}

\item \label{thm.min-plus1} Over the semiring $\bb{R}$ the system $F$ has a root iff the Macaulay min-plus linear system $\Ml_N \tp \vec{y} = \Mr_N \tp \vec{y}$ for
$$
N = (n+2)\left(d_1 + \ldots + d_k\right)
$$
has a solution.

\item \label{thm.min-plus2} Over the semiring $\bb{R}_\infty$ the system $F$ has a root iff the non-homogeneous Macaulay min-plus linear system $\Ml_N \tp (0,\vec{y}) = \Mr_N \tp (0,\vec{y})$ for
$$
N = \poly(n,k,d) \left(4d\right)^{\min(n,k)}
$$
has a solution.

\end{enumerate}
\end{theorem}

As in the tropical case an infinite version of the min-plus dual Nullstellensatz follows.

\begin{corollary}
Consider the system of min-plus polynomials $F = \{f_1=g_1, \ldots, f_k=g_k\}$ of $n$ variables.
The system $F$ has a root iff the infinite Macaulay min-plus linear system with the pair of matrices $(\Ml, \Mr)$
has a solution.
The result holds for both $\bb{R}$ and $\bb{R}_\infty$ semirings.
\end{corollary}

We provide examples showing that our bounds on $N$ are qualitatively tight.
Namely for the semiring $\bb{R}$ we construct a family $F$ of $(n+1)$ tropical (or min-plus)
polynomials of $n$ variables and of degree $d$ such that $F$ has no root, but the Macaulay tropical (or min-plus) linear system for $N = (d-1)(n-1)$ has a solution.
For the semiring $\bb{R}_\infty$ for any $d>1$ we construct a system $F$ of $n+1$ tropical (or min-plus) polynomials of $n+1$ variables and
of degree $d$ such that $F$ has no root, but the Macaulay tropical (or min-plus) linear system for $N = d^{n-1} - 1$ has a solution.

We note that quantitatively there is a room for improvement between our lower and upper bounds on $N$.
The gap is more noticeable in the case of the semiring $\bb{R}$. Assuming for the sake of simplicity that $n \approx k$
our upper bound gives $N \sim d n^2$ and our lower bound gives $N \sim dn$.
Thus we can formulate the following open problem.

\begin{openproblem}
Close the gap between upper and lower bounds on $N$ in the tropical Nullstellensatz.
\end{openproblem}

Next we establish the Nullstellensatz
in a more standard primary form.

We start with a more intuitive min-plus Nullstellensatz.
\begin{theorem}[Min-Plus Primary Nullstellensatz] \label{thm.min-plus-primary}
Consider a system of min-plus polynomials $F = \{(f_1,g_1), \ldots, (f_k,g_k)\}$ in $n$ variables. Denote by $d_i$ the degree of the polynomial $(f_i, g_i)$ and let $d = \max_i d_i$.
In algebraic combinations $(f,g)$ of the polynomials in $F$ we allow to use not only polynomials $(f_i,g_i)$,
but also $(g_i, f_i)$.
\renewcommand*{\theenumi}{\thetheorem(\roman{enumi})}%
  \renewcommand*{\labelenumi}{(\roman{enumi})}%
\begin{enumerate}
\item Over the semiring $\bb{R}$ the system $F$ has no root
iff we can construct an algebraic min-plus combination $(f, g)$ of $F$ with degree at most
$$
N = (n+2)\left(d_1 + \ldots + d_k\right)
$$
such that for each monomial $m = x_1^{\tp j_1} \tp \ldots \tp x_n^{\tp j_n}$
its coefficient in $f$ is greater than its coefficient in $g$. 

\item Over the semiring $\bb{R}_\infty$ the system $F$ has no root iff we can construct an algebraic combination $(f,g)$ of $F$ with degree at most
$$
N = \poly(n,k,d) \left(4d\right)^{\min(n,k)}
$$
such that for each monomial $m = x_1^{\tp j_1} \tp \ldots \tp x_n^{\tp j_n}$ its coefficient in $f$ is greater than its coefficient in $g$ and with an additional property that the constant term in $g$ is finite.
\end{enumerate}
\end{theorem}

For the tropical case we will need the
following definition.

\begin{definition}
For a system of tropical polynomials $F = \{f_1, \ldots, f_k\}$ and tropical monomials
$m_1, \ldots, m_K$ the algebraic combination
$$
g = \bigoplus_{j=1}^{K} g_j,
$$
where
$$
g_j =  m_j \tp f_{i_j},
$$
is called nonsingular if the following two properties hold:
\begin{itemize}
\item for each monomial $m$ of $g$ there is a (unique) $1\le l(m) \le K$ such that the coefficient of $m$ in the polynomial $g_{l(m)}$ is less than the coefficients of $m$ in all other polynomials $g_j$ for $j\neq l(m)$;
\item for different $m$ and $m'$ we have $l(m)\neq l(m')$.
\end{itemize}
\end{definition}

Now we can formulate the tropical Nullstellensatz in a primary form.
\begin{theorem}[Tropical Primary Nullstellensatz] \label{thm.tropical-primary}
Consider a system of tropical polynomials $F = \{f_1, \ldots, f_k\}$ in $n$ variables. Denote by $d_i$ the degree of the polynomial $f_i$ and let $d = \max_i d_i$.
\renewcommand*{\theenumi}{\thetheorem(\roman{enumi})}%
  \renewcommand*{\labelenumi}{(\roman{enumi})}%
\begin{enumerate}
\item The system $F$ has no root over $\bb{R}$
iff there is a nonsingular algebraic combination $g$ of $F$ with degree at most
$$
N = (n+2)\left(d_1 + \ldots + d_k\right)
$$

\item The system $F$ has no root over $\bb{R}_\infty$
iff there is a nonsingular algebraic combination $g$ of $F$ with degree at most
$$
N = \poly(n,k,d) \left(4d\right)^{\min(n,k)}
$$
and with a finite constant monomial.
\end{enumerate}
\end{theorem}

For the proofs of the last two theorems we use the following min-plus and tropical linear duality.

\subsection{Linear Duality}

We prove the following result on the min-plus linear duality.

\begin{lemma} \label{lem:min-plus_linear_duality}
Let $A,B \in \bb{R}_{\infty}^{m\times n}$ be two matrices.

For any subset $S \subseteq \{1,\ldots, n\}$ exactly one of the following two statements is true.
\begin{enumerate}
\item There is a solution to $A \tp \vec{x} \leq B \tp \vec{x}$ with finite coordinates $x_i$ for $i \in S$.
\item There is a solution to $B^T \tp \vec{y} < A^T \tp \vec{y}$ such that for some $i \in S$ the $i$-th coordinates of the vector $B^T \tp \vec{y}$ is finite.
\end{enumerate}

For any subset $S \subseteq \{1,\ldots, n\}$ exactly one of the following two statements is true.
\begin{enumerate}
\item There is a solution to $A \tp \vec{x} \leq B \tp \vec{x}$ such that for some $i \in S$ the coordinate $x_i$ is finite.
\item There is a solution to $B^T \tp \vec{y} < A^T \tp \vec{y}$ such that the $i$-th coordinates of the vector $B^T \tp \vec{y}$ are finite for all $i \in S$.
\end{enumerate}
\end{lemma}

The proof of this lemma is based on the connection of min-plus linear systems with mean payoff games established in~\cite{akian12mean_payoff}.
Though the proof is rather simple as soon as one has this connection, we are not aware of the claim and the proof of this result in the literature.

As a simple corollary of this lemma we show the following clean formulation of the min-plus linear duality.

\begin{corollary} \label{cor:min-plus_linear_duality}
For two matrices $A,B \in \bb{R}^{m\times n}$ exactly one of the following two statements is true.
\begin{enumerate}
\item There is a solution to $A \tp \vec{x} \leq B \tp \vec{x}$ over $\bb{R}$.
\item There is a solution to $B^T \tp \vec{y} < A^T \tp \vec{y}$ over $\bb{R}$.
\end{enumerate}
For two matrices $A,B \in \bb{R}_{\infty}^{m\times n}$ exactly one of the following two statements is true.
\begin{enumerate}
\item There is a solution $\vec{x} \neq (\infty, \ldots, \infty)$ to $A \tp \vec{x} \leq B \tp \vec{x}$.
\item There is a finite solution to $B^T \tp \vec{y} < A^T \tp \vec{y}$.
\end{enumerate}
For two matrices $A,B \in \bb{R}_{\infty}^{m\times n}$ exactly one of the following two statements is true.
\begin{enumerate}
\item There is a finite solution to $A \tp \vec{x} \leq B \tp \vec{x}$.
\item There is a solution $\vec{y} \neq (\infty, \ldots, \infty)$ to $B^T \tp \vec{y} < A^T \tp \vec{y}$.
\end{enumerate}
\end{corollary}

Since the corollary follows from Lemma~\ref{lem:min-plus_linear_duality} almost immediately,
we present the proof here.

\begin{proof}
For matrices $A,B \in \bb{R}^{m\times n}$ with finite entries we can use the first part of Lemma~\ref{lem:min-plus_linear_duality} with $S = \{1,\ldots, n\}$. Then the first statement in the lemma coincides with the first statement in the corollary. The second statement in the lemma is equivalent to the second statement in the corollary. Indeed, if there is a finite solution to the system $B^T \tp \vec{y} < A^T \tp \vec{y}$ then clearly for this $\vec{y}$ all the coordinates of the vector $B^T \tp \vec{y}$ are finite. In the reverse direction, if there is a solution $\vec{y}$ such that $B^T \tp \vec{y}$ has a finite coordinate, then the vector $\vec{y}$ itself has a finite coordinate. Since all the entries in $A$ and $B$ are finite, then all the coordinates of $A^T \tp \vec{y}$ and $B^T \tp \vec{y}$ are finite. If in the vector $\vec{y}$ there are infinite coordinates we can replace them by large enough finite numbers in such a way that the vectors $A^T \tp \vec{y}$ and $B^T \tp \vec{y}$ do not change. The resulting vector is a finite solution of the system $B^T \tp \vec{y} < A^T \tp \vec{y}$.

For the second part of the corollary let $S = \{1,\ldots, n\}$ and apply the second part of Lemma~\ref{lem:min-plus_linear_duality}.
Then the first statement in Lemma~\ref{lem:min-plus_linear_duality} is equivalent to the first statement in the corollary.
To see that the equivalence holds also for the second statements note that if for some $\vec{y}$ all the coordinates of $B^T \tp \vec{y}$
are finite, then we can assume that all the coordinates of $\vec{y}$ are also finite.
Indeed, if there are infinite coordinates in $\vec{y}$ we can just set them to constants large enough not to change the value of the vector $B^T \tp \vec{y}$.

The last part of the corollary can be shown analogously by letting $S = \{1,\ldots, n\}$ and applying the first part of Lemma~\ref{lem:min-plus_linear_duality}.

\end{proof}

We show a similar result for the tropical duality.
\begin{lemma} \label{lem:tropical_linear_duality}
Let $A \in \bb{R}_{\infty}^{m\times n}$ be a matrix.

For any subset $S \subseteq \{1,\ldots, n\}$ exactly one of the following two statements is true.
\begin{enumerate}
\item There is a solution to $A \tp \vec{x}$ with finite coordinates $x_i$ for $i \in S$.
\item There is $\vec{z}$ such that in each row of $A^T \tp \vec{z}$ the minimum is attained only once or is equal to $\infty$, for each two rows with the finite minimum the (unique) minimums are in different columns and
for some $i \in S$ the $i$-th coordinate of $A^T \tp \vec{z}$ is finite.
\end{enumerate}

For any subset $S \subseteq \{1,\ldots, n\}$ exactly one of the following two statements is true.
\begin{enumerate}
\item There is a solution to $A \tp \vec{x}$ such that for some $i \in S$ the coordinate $x_i$ is finite.
\item There is $\vec{z}$ such that in each row of $A^T \tp \vec{z}$ the minimum is attained only once or is equal to $\infty$, for each two rows with the finite minimum the minimums are in different columns and the $i$-th coordinates of $A^T \tp \vec{z}$ are finite for all $i \in S$.
\end{enumerate}
\end{lemma}

This result can be proven either through a reduction to min-plus linear systems, or through the analysis of~\cite{Grigoriev13complexity}.
We give a proof through the reduction to min-plus linear systems in Section~\ref{sec:linear_duality}.

Just like in the case of min-plus linear systems we can show the following corollary.

\begin{corollary} \label{cor:tropical_linear_duality}
For a matrix $A \in \bb{R}^{m\times n}$ exactly one of the following two statements is true.
\begin{enumerate}
\item There is a solution to $A \tp \vec{x}$ over $\bb{R}$.
\item There is $\vec{z}\in \bb{R}^{m}$ such that in each row of $A^T \tp \vec{z}$ the minimum is attained only once
and for each two rows the minimums are in different columns.
\end{enumerate}
For a matrix $A \in \bb{R}_{\infty}^{m\times n}$ exactly one of the following two statements is true.
\begin{enumerate}
\item There is a solution $\vec{x} \neq (\infty, \ldots, \infty)$ to $A \tp \vec{x}$.
\item There is a finite $\vec{z}$ such that in each row of $A^T \tp \vec{z}$ the minimum is attained only once
and for each two rows the minimums are in different columns.
\end{enumerate}
For a matrix $A \in \bb{R}_{\infty}^{m\times n}$ exactly one of the following two statements is true.
\begin{enumerate}
\item There is a finite solution to $A \tp \vec{x}$.
\item There is $\vec{z}\neq (\infty, \ldots, \infty)$ such that in each row of $A^T \tp \vec{z}$ the minimum is attained only once
or is equal to $\infty$ and for each two rows with the finite minimum the minimums are in different columns.
\end{enumerate}
\end{corollary}

The proof of this corollary is completely analogous to Corollary~\ref{cor:min-plus_linear_duality}.

\subsection{Tropical vs. Min-plus}

We also establish the connection between tropical and min-plus polynomial systems.

\begin{lemma}
Over both $\bb{R}$ and $\bb{R}_\infty$ given a system of tropical polynomials we can construct a system of min-plus polynomials
over the same set of variables and with the same set of roots.
\end{lemma}

In the opposite direction we do not have such a simple connection, but we can still prove the following lemma.
\begin{lemma}
Over both $\bb{R}$ and $\bb{R}_{\infty}$ for any system of min-plus polynomials $F$ in $n$ variables
there is a system of tropical polynomials $T$ in $2n$ variables
and an injective linear transformation $H \colon \bb{R}_{\infty}^n \to \bb{R}_{\infty}^{2n}$ such that
the image of the set of  roots of $F$ coincides with the set of roots of $T$.
\end{lemma}

The proof of this lemma follows the lines of the proof of the analogous statement for the case of linear polynomials in~\cite{GP13Complexity}.

\section{Tropical and Min-plus Dual Nullstellensatz} \label{sec:tropical_dual}

Throughout the whole section
we assume that we are given a system of tropical polynomials $F = \{f_1, \ldots, f_k\}$ in $n$ variables $\vec{x}=(x_1, \ldots, x_n)$.

The proofs of Theorems~\ref{thm:main} and~\ref{thm.min-plus} are analogous. We present the proof of Theorem~\ref{thm:main} (which is more intuitive) and specify what should be changed to obtain the proof of Theorem~\ref{thm.min-plus}.

This section is organized as follows.
In Subsection~\ref{sec:null_prelim} we introduce required notation and show preliminary results.
In Subsections~\ref{sec:env_polytope} and~\ref{sec:singular_facet} we give a proof of Theorems~\ref{thm:main1} and~\ref{thm.min-plus1}.
In Subsection~\ref{sec:examples} we provide some examples illustrating the difficulties behind the proof.
In Subsection~\ref{sec:infinite case} we prove Theorems~\ref{thm:main2} and~\ref{thm.min-plus2}.
Finally, in Subsection~\ref{sec:lower_bounds} we show that the upper bounds in Theorems~\ref{thm:main} and~\ref{thm.min-plus} are tight.

\subsection{Preliminary definitions and results} \label{sec:null_prelim}

\paragraph{Geometrical interpretation of tropical polynomials.}
All functions $\phi \colon \bb{Z}^n \to \bb{R}$ we consider in this section are partial, that is they are defined on some subset $\Dom(\phi) \subseteq \bb{Z}^n$.

\begin{definition}For two functions $\phi, \psi \colon \bb{Z}^n \to \bb{R}$ let $D = \Dom \phi \cap \Dom \psi$ and consider $t \in \bb{R}$ (if there is one) such that
\begin{enumerate}
  \item for all
$\vec{x} \in D$ we have $\phi(\vec{x}) + t \leq \psi(\vec{x})$;
  \item there is $\vec{x} \in D$ such that $\phi(\vec{x}) + t = \psi(\vec{x})$.
\end{enumerate}
It is easy to see that $t$ is unique provided it exists.
We denote the set of points satisfying property~$2$ by $\Sing(\phi,\psi)$ and call them \emph{singularity points} for the pair $(\phi,\psi)$.
If such $t$ does not exist we let $\Sing(\phi,\psi) = \emptyset$. We say that
$\phi$ is \emph{singular} to $\psi$ iff $|\Sing(\phi,\psi)| \neq 1$.
\end{definition}

Geometrically, (if $\Dom \phi \cap \Dom \psi$ is a finite set) $\phi$ is singular to $\psi$ if either the domains of $\phi$ and $\psi$ do not intersect, or if we can adjust the graph of $\phi$ in $\bb{R}^{n+1}$ space along the $(n+1)$-th coordinate in such
a way that this graph lies below the graph of $\psi$ and has with it at least two common points. 
For a function $\phi$ we denote by $G(\phi)$ the graph of the function in $\bb{R}^{n+1}$.

Note that the notion of singularity is nonsymmetric. It might be that $\phi$ is singular to $\psi$, but $\psi$ is not singular to $\phi$.

The following lemma follows directly from the definition.
\begin{lemma} \label{lem:singularity}
We have $\vec{x} \in \Sing(\phi,\psi)$ iff $\vec{x}$ minimizes the function $\psi - \phi$ (on its domain).
Or, equivalently, iff for all $\vec{y}$ we have
$$
\phi(\vec{x}) - \phi(\vec{y}) \geq \psi(\vec{x}) - \psi(\vec{y}).
$$
\end{lemma}

In this paper we consider rows of the matrix $M_N$, solutions to $M_N \tp \vec{y}$, coefficients of $f_i$'s. All of them constitute vectors
$\vec{a}$ which coordinates are labeled by $I \in D$ for some $D \subseteq \bb{Z}^n_{+}$, that is by vectors with integer non-negative coordinates.
With a vector $\vec{a}$ we associate a function $\phi_{\vec{a}} \colon \bb{Z}^n \to \bb{R}$
letting $\phi_{\vec{a}}(I) = a_I$ for $I \in D$ and $a_I \neq \infty$ and leaving $\phi_{\vec{a}}(I)$ undefined otherwise.
When this vector is the vector of the coefficients of a polynomial $f$ we shortly denote the resulting function by $\phi_f$.
When a polynomial $f$ is one of the polynomials $f_i \in F$ we simplify the notation further to $\phi_i$.
Note that due to the definition of $M_N$ if $\vec{r}$ is a row of $M_N$ labeled by $(J,i)$ then
$\phi_{\vec{r}}(I) = \phi_i(I - J)$.

In what follows we reserve Greek letters for the functions representing the coefficients of polynomials and entries of Macaulay matrix
to distinguish them from the functions $f_i$'s.

The motivation for our notion of singularity is that it captures the
solvability of tropical polynomials.

\begin{lemma} \label{lem:singular_vs_solution}
A vector $\vec{a}=\{a_I\}_{|I|\leq N}$ over $\bb{R}$ or $\bb{R}_{\infty}$ is a solution to a tropical linear polynomial $\min_I \{y_I + r_I\}$ corresponding to the vector $\vec{r} = \{r_I\}_{|I|\leq N}$ iff the function $- \phi_{\vec{a}}$ is singular to $\phi_{\vec{r}}$.
\end{lemma}

\begin{proof}
Consider an arbitrary vector $\vec{r}$ over $\bb{R}$ or $\bb{R}_{\infty}$ and the corresponding tropical linear polynomial.
The vector $\vec{a}$ is a root
of this linear polynomial if the minimum in $\{\phi_{\vec{a}}(I) + \phi_{\vec{r}}(I)\}_{I}$ is attained at least twice or is equal to $\infty$. This minimum is $\infty$ iff $\Dom(\phi_{\vec{a}}) \cap \Dom(\phi_{\vec{r}}) = \emptyset$ and thus $\Sing(-\phi_{\vec{a}}, \phi_{\vec{r}}) = \emptyset$.
If the minimum is finite let $t$ be the minimal number such that $\phi_{\vec{a}}(I) + \phi_{\vec{r}}(I) + t \geq 0$ for all $I$. Then $\phi_{\vec{a}}(I) + \phi_{\vec{r}}(I) + t = 0$
equals zero for at least two different $I$'s.
This means that $- \phi_{\vec{a}}(I) - t \leq \phi_{\vec{r}}(I)$ and equality holds for at least two points.
Thus the function $- \phi_{\vec{a}}$ is singular to $\phi_{\vec{r}}$.

The proof in the opposite direction follows the same lines.
\end{proof}

In particular, a vector $\vec{y}$ is a solution to $M_N \tp \vec{y}$ iff $- \phi_{\vec{y}}$ is singular to all $\phi_{\vec{r}}$, where $\vec{r}$ is a row of $M_N$.

The difference between the semirings $\bb{R}$ and $\bb{R}_{\infty}$ is that over $\bb{R}$ we have $\Dom \phi_{\vec{y}} = \{I |\  |I| \leq N\}$.

Now let $\vec{r}$ be a vector of coefficients of a tropical polynomial $f$, that is $r_{I}$ is the coefficient of the monomial $\vec{x}^I$ in $f$.
Then a root of the polynomial $f$ is a vector
$\vec{x} = (x_1, \ldots, x_n)$ and $\vec{y}$ described in the previous paragraph in this case is given by
$$y_{I} = \vec{x}^{\tp I} =  \sum_j i_j x_j = \langle\vec{x}, I \rangle,$$ that is by the (classical) inner product of vectors $\vec{x}$ and $I$. 
Thus in this case 
$\phi_{\vec{y}}(I) = \langle\vec{x}, I \rangle$ is a partial linear function 
(note, that some coordinates $x_j$ of $\vec{x}$ might be infinite and then $y_I = \infty$ once $i_j\neq 0$ for at least one infinite coordinate $x_j$), 
which graph is a part of a hyperplane in $(n+1)$-dimensional space.
Note that here we assume that $0 \cdot \infty = 0$.
We introduce the notation $\chi_{\vec{x}} = -\phi_{\vec{y}}$
and we say that $\chi_{\vec{x}}$ is a \emph{partial hyperplane}.
Over $\bb{R}$ the function $\chi_{\vec{x}}$ corresponds to an ordinary hyperplane.
Thus, from Lemma~\ref{lem:singular_vs_solution} we get the following result.

\begin{lemma}
A vector $\vec{x} \in \bb{R}^n_{\infty}$ is a root to $f$ iff
the partial hyperplane $\chi_{\vec{x}}$ is singular to the function $\phi_f$.
\end{lemma}

In particular, the system of polynomials $F$ has a root over $\bb{R}_{\infty}$ iff there is a partial hyperplane
singular to $\phi_i$ for all $i = 1, \ldots, k$.

As a result we have that, if there is a partial hyperplane singular to all $\phi_i$ for all $i= 1, \ldots, k$, then it clearly provides a solution to
$M_N$. This repeats the proof of the simple direction of the tropical dual Nullstellensatz theorem.
What we need to show for the opposite direction is that if there is some function singular to all translations of all $\phi_i$'s within some simplex $|I| \leq N$,
then there is also a singular partial hyperplane.

For the proof of Theorem~\ref{thm:main} it is convenient to use the language of polytopes.
We summarize it in the next definition.

\begin{definition}
To switch to polytope notation for a polynomial $f \in F$ we consider the graph $G(\phi_f) = \{(I, \phi_f(I)) \mid |I| \leq N,\ \phi_f(I)\neq \infty\}$ of the function $\phi_f$ and along with each point $(I, \phi_f(I))$ we consider all points $(I, t)$ above it, that is such that $t > \phi_f(I)$.
We take the convex hull in $\bb{R}^{n+1}$ of all these points and call the resulting polytope $P(f)$
the \emph{(extended) Newton polytope} of $f$. 
For the given system $F$ of polynomials $f_{1}, \ldots, f_{k}$ we denote the resulting convex polytopes
by $P_1, \ldots, P_k$.
We note that this construction is quite standard~\cite{IMS2009tropical,RGST05first_steps,sturmfels02equations}.
By the \emph{bottom} of $P(f)$ we denote the set of points $\vec{x} = (x_1, \ldots, x_n, x_{n+1}) \in P(f)$ such that
there are no points of $P(f)$ below them, that is for any $\eps > 0$ we have that $(x_1, \ldots, x_n, x_{n+1}- \eps) \notin P(f)$.
Note that the bottom of $P(f)$ can be considered as a graph of a partial function on $\bb{R}^n$.
We denote the restriction of this function to $\bb{Z}^n$ by $\beta_P \colon \bb{Z}^n \to \bb{R}$. 
For the case of polytopes $P_i$ we shorten this notation to $\beta_i$.
It is not hard to see geometrically that a partial hyperplane is singular to $\phi_f$ iff it is singular to $\beta_P$.
This is not necessarily true for an arbitrary function $\phi_{\vec{a}}$ instead of a hyperplane.
\end{definition}

%To bring together the functional language and the language of polytopes we introduce two more notations.
%For a function $\phi$ we denote by $G(\phi)$ the graph of the function in $\bb{R}^{n+1}$\nb{move up?}.
%For the case of the functions $\phi_i$ to make the notation more intuitive instead of $G(\phi_i)$
%we write $G(f_i)$. 

\begin{remark}
We note that in the paper~\cite{Grigoriev12null} the conjecture on the tropical dual Nullstellensatz was considered
not for the original Macaulay matrix, but for the Macaulay matrix in which we already switch to the convex hulls of the graphs of polynomials in $F$,
that is instead of values of functions $\phi_i$ rows of Macaulay matrix contained graphs of functions $\beta_{P_i}$.
Our proofs works for both settings, but we consider the original Macaulay matrix being more natural.
\end{remark}

\begin{remark} \label{rk:colors}
For the min-plus case we analogously associate to
each polynomial $(f_i, g_i) \in F$ the function $\phi_i$.
We let $\phi_i(I)$ to the minimum of the coefficients of the monomial $\vec{x}^I$
in $f_i$ and $g_i$. Additionally we introduce colors to the points of the graphs $G(\phi_i)$.
If $\phi_i(I)$ is equal to the coefficient of $\vec{x}^I$ in $f_i$, then we color $(I,\phi_i(I))$ in black,
and if $\phi_i(I)$ is equal to the coefficient of $\vec{x}^I$ in $g_i$, then we color $(I,\phi_i(I))$ in white.
Note that we allow the same point $I$ be labeled by both colors simultaneously.
The notion of singularity changes in that now we require that there are either no singular points, or at least one black singular point and at least one white singular point.
Note, however, that we are satisfied if there is only one singular point, but labeled with both colors.
Analogous analysis shows that this notion of singularity captures the notion of min-plus solvability
in the same way as in the tropical case.

Newton polytopes $P_i$ are introduced in the same way as before (in the construction of the polytope we ignore the colors).
But note that now some points of $P_i$ are labeled by colors. In particular,
all vertices of polytopes are labeled.
\end{remark}

\paragraph{Convex polytopes.}

A convex polytope $P$ in $n$-dimensional space can be specified by a set of (classical) linear functions $E_1(\vec{x}), \ldots, E_{l}(\vec{x}), L_1(\vec{x}), \ldots, L_k(\vec{x})$, where $\vec{x} \in \bb{R}^n$:
$P$ is the set of points
$\vec{x} \in \bb{R}^{n}$ such that $E_{i}(\vec{x}) = 0$ for all $i=1,\ldots, l$ and $L_{i}(\vec{x}) \geq 0$ for all $i=1, \ldots, k$.
We assume that none of $L_i(\vec{x})$ evaluates to $0$ on the set $\{\vec{x} | \forall i E_i(\vec{x}) = 0 \}$.
Any face of a polytope can be specified by a nonempty set $S \subseteq \{1, \ldots, k\}$.
The face corresponding to $S$ is the set of points $\vec{x} \in P$ such that $L_i(\vec{x})=0$ for all $i\in S$.

%Thus we will always assume that the polytopes are closed, that.
%The same applies to the faces of polytopes, by default they are considered to be closed.
The boundary of the polytope is the union of all its faces.
The interior $\inter{P}$ of the polytope $P$ is the set of its non-boundary points.

%We will use an algebraic notation on polytopes in $n$-dimensional space in terms of the Minkowski sum.
%That is
For polytopes $P_1$ and $P_2$ we denote by $P_{1} + P_2$ the Minkowski sum of these polytopes.
For natural $k$ we use the notation $kP = P + \ldots + P$, where there are $k$ summands on the right-hand side.
For an $n$-dimensional vector $\vec{\alpha}$ we denote by $P + \vec{\alpha}$ the translation of $P$ by the vector $\vec{\alpha}$.
That is,
$$
P + \vec{\alpha} = \{\vec{x} + \vec{\alpha} \mid \vec{x} \in P\}.
$$
By the homothety with a center $\vec{x} \in \bb{R}^n$
and a coefficient $\lambda > 0$ we denote the following bijective transformation $h_{\vec{x}}^\lambda$ of the space $\bb{R}^n$:
the point $\vec{y} \in \bb{R}^n$ is sent to the point $h_{\vec{x}}^\lambda=\vec{x} + \lambda (\vec{y} - \vec{x})$.
Note that $kP$ is an image of $P$ under the homothety $h^k_{\vec{0}}$.
It is well known that translations and homotheties with the composition operation form a group called dilation group. In particular, an arbitrary composition of translations and homotheties results in a homothety or a translation. Below we will use this fact without mentioning.

\begin{definition}
Consider a polytope $P \subseteq \bb{R}^n$, a set of points $Q\subseteq \bb{R}^n$ and a point $\vec{x}$ on the boundary of $P$.
We say that $Q$ \emph{touches} $P$ at $\vec{x}$ iff
\begin{enumerate}
  \item $Q \subseteq P$;
  \item $\vec{x} \in Q$;
  \item if $Q$ contains a point $\vec{y}$ on the boundary of $P$, then $\vec{y}$ lies in a face of $P$ containing $\vec{x}$.
\end{enumerate}
\end{definition}

Below we collect some facts we will need about the structure of convex polytopes.
Though they are simple and intuitive we give the proofs of them for the sake of completeness.

\begin{lemma} \label{lem:boundary}
Let $P$ be a convex polytope and let $\vec{x},\vec{y},\vec{z}$ be three distinct points in $P$
lying on the same line in the specified order.
Then if $\vec{y}$ belongs to some face of $P$ then $\vec{x}$ also belongs to the same face of $P$.
\end{lemma}

\begin{proof}
Suppose on the contrary that $\vec{y}$ belongs to some face and $\vec{x}$ does not. Then there is some inequality $L$ among linear inequalities defining
$P$ such that $L(\vec{x}) > 0$ and $L(\vec{y})=0$. 
The restriction of values of $L$ to the line containing $\vec{x}$, $\vec{y}$ and $\vec{z}$
is a linear function and hence, clearly, $L(\vec{z})<0$. Therefore $\vec{z}$ is not in $P$ and thus we have a contradiction.
\end{proof}

\begin{corollary} \label{cor:boundary}
Let $P$ be a convex polytope and let $\vec{x},\vec{y},\vec{z},\vec{t}$ be four distinct points in it
lying on the same line in the specified order.
Then $\vec{y}$ belongs to some face of $P$ iff $\vec{z}$ belongs to the same face.
\end{corollary}

\begin{proof}
Just apply Lemma~\ref{lem:boundary} to the points $\vec{y}, \vec{z}, \vec{t}$ and to the points $\vec{z}, \vec{y}, \vec{x}$.
\end{proof}

\begin{lemma} \label{lem:homothety}
Let $P$ be a convex polytope and let $\vec{x}$ be a point in $P$. Let $P^\prime$ be an image of 
$P$ under $h^{\lambda}_{\vec{x}}$ for $\lambda > 1$.
If $P$ contains a point on some face of $P^{\p}$ then this face contains $\vec{x}$.
\end{lemma}

\begin{proof}
Let $\vec{y}$ be a point of $P$. Then the point 
$$
\vec{z} = 
h^\lambda_{\vec{x}}(\vec{y}) = 
\vec{x} - \lambda (\vec{y} - \vec{x})
%=\vec{y} + (\lambda - 1) (\vec{y} - \vec{x})
$$
lies in $P^\p$
and $\vec{x}, \vec{y}, \vec{z}$ lie on the same line in the specified order.
Thus by Lemma~\ref{lem:boundary} if $\vec{y}$ is on some face of $P^\p$ then $\vec{x}$ is also on this face.
\end{proof}

\subsection{The enveloping polytope} \label{sec:env_polytope}

%We consider the Newton polytopes $P_1, \ldots, P_k$ corresponding to $f_1, \ldots, f_k$.
The key idea of the proof of Theorem~\ref{thm:main} is to consider a large extended Newton polytope $P_0$ ``enveloping'' all polytopes $P_1, \ldots, P_k$ corresponding to tropical polynomials $f_1, \ldots, f_k$.
The main property of $P_0$ we will ensure in this subsection is that for each point $\vec{x}$ on its bottom and
for any $i$ there is a translation $P_i + \vec{\alpha}$ of the polytope $P_i$ such that $P_i + \vec{\alpha}$ touches $P_0$ at $\vec{x}$ (Lemma~\ref{lem:touching}).
%\begin{enumerate}
%  \item it lies completely inside of $P_0$;
%  \item it has $\vec{x}$ as a vertex;
%  \item it does not intersect any facets of $P_0$ other than adjacent to $\vec{x}$.
%\end{enumerate}
Using this property in the next subsection we will show that 
for a solution $\vec{a}$ of the tropical linear system $M_N \tp \vec{y}$ there is a singular point in $\Sing(\vec{a}, \beta_{P_0})$ such that the facet of $P_0$ containing the point $(\vec{a}, \beta_{P_0}(\vec{a}))$
gives a root to the system $F$ (Lemma~\ref{lem:facet}).

It turns out that for $P_0$ we can just take the Minkowski sum of $P_1, \ldots, P_k$ multiplied by a large enough number.
We just let
\begin{equation} \label{eq.polytope}
P_0 = (n+2)\cdot \left(P_1 + \ldots + P_k\right).
\end{equation}
%It is clear that for each $P_i$ we have that $P_0$ can be represented as the union
%of the translations of $P_i$ by the real vectors. However we will need that all
%integer points
%can be represented by integer translations of vertices of $P_i$ (we will actually need slightly more).

To ensure the desired property of $P_0$ we need the following general fact on convex polytopes.

\begin{lemma} \label{lem:discretization}
Let $P$ be an $n$-dimensional convex polytope and let $P^{\p} = (n+2) P$.
Then for each point $\vec{x} \in \inter{P}^{\p}$ there is a translation $P + \vec{\alpha}$ of $P$ and the homothety $h^{n+2}_{\vec{y}}$ mapping $P + \vec{\alpha}$ to $P^{\p}$ with the following properties:
\begin{enumerate}
\item the center $\vec{y}$ of the homothety lies in $\inter{P}^\p$;
\item $\vec{x}$ is a vertex of $P + \vec{\alpha}$.
\end{enumerate}
\end{lemma}

It is easy to see that the first property is equivalent to the fact that $P + \vec{\alpha} \subseteq \inter{P}^\p$,
but the current form of the lemma will be more convenient for us.

The main tools in the proof of this lemma are Caratheodory's Theorem, the notion of the center of mass and homothety transformations.

\begin{proof}

We first give a proof sketch and then proceed to the detailed formal proof.
Since $\vec{x}$ is in $P^\p$ it lies in some simplex $S^\p$ generated by $n+1$ vertices of $P^\p$.
For $S^\p$ we consider each of its vertices and make a homothety with the center in it and the coefficient $(n+1)/(n+2)$.
The resulting $(n+1)$ simplices cover all $S^\p$ (even with overlap).
So $\vec{x}$ lies in one of them, say in the one determined by the vertex $v^\p_1$. Then we can consider the translation $S+\vec{\alpha}$
of the simplex $S$ which is $(n+2)$ times smaller than $S^\p$
such that its vertex corresponding to $v^\p_1$ is mapped into $\vec{x}$. Then $S + \vec{\alpha}$ lies in $S^\p$.
Now we can consider $P^\p$ and note that $P+\vec{\alpha}$ is in $P^\p$.

Now we give a formal proof following the outline above.
Since $\vec{x}$ is a point in the convex polytope $P^{\p}$ it lies in the convex hull of its vertices.
By Caratheodory's Theorem there are $n+1$ vertices $v^\p_1, \ldots, v^\p_{n+1}$ of $P^{\p}$ such that $\vec{x} \in \Conv\{v^\p_1, \ldots, v^\p_{n+1}\}$.
We denote the latter simplex by $S^{\p}$ and
denote the corresponding vertices of $P$ by $v_1, \ldots, v_{n+1}$.

%For each vertex $v_i$ of $S$ we consider the homothety transformation of $S$ with the center in $v_i$ and coefficient $(n+1)/(n+2)$.
%We denote the image of $S$ under this transformation by $S_i$. We will show that each

Let $w_1, \ldots w_{n+1}$ be the barycentric coordinates of $\vec{x}$ with respect to $v^\p_1, \ldots, v^\p_{n+1}$, that is $w_i \geq 0$ for all $i$, $\sum_i w_i = 1$ and
$$
\vec{x} = \sum_i w_i v^\p_i.
$$
Without loss of generality let $w_1$ be the largest among $w_i$. Then $n w_1 \geq w_2 + \ldots + w_{n+1}$.
Let
$$
v^\p = \frac{1}{\sum_{i=2}^{n+1} w_i} \sum_{i=2}^{n+1} w_i v^\p_i.
$$
Then $v^\p \in \Conv\{v^\p_2, \ldots v^\p_{n+1}\}$, there is a relation
$$
\vec{x} = w_1 v_1^\p + \left(\sum_{i=2}^{n+1} w_i\right) v^\p
$$
and thus the points $v^\p_1$, $\vec{x}$ and $v^\p$
are on the same line. Moreover, $|\vec{x} - v^\p_{1}| \leq n |v^\p - \vec{x}| < (n+1)|v^\p - \vec{x}|$ (observe that $|v^\p - \vec{x}|$ is nonzero since $w_1$ is nonzero being the largest weight).
Consider a point $\vec{y}$ on the same line between the points $\vec{x}$ and $\vec{v}^\p$ and such that $(n+1)|\vec{x} - \vec{y}| =  |v_1^\p - \vec{x}|$.

Now consider the polytope $P$ and consider its translation $P + \alpha$ by which $v_1$ is mapped to $\vec{x}$. The homothety $h_{\vec{y}}^{n+2}$ sends $\vec{x}$ to $v_1$ and since $(n+2)P$ and $P^\p$ are equal, the image of $P + \alpha$ under this homothety is $P^\p$.

%Consider the homothety transformation of $S^\p$ with the center in $v^\p_1$
%and coefficient $(n+1)/(n+2)$.
%Denote the image of $S^{\p}$ by $S_1$. % and the image of $\Span\{v^\p_2, \ldots v^\p_{d+1}\}$ by $\pi$.
%Then from above we get that $\vec{x} \in S_1$.
%Now we can consider $S = \Conv \{v_1, \ldots, v_{n+1}\}$
%and consider its translation $S+\vec{\alpha}$ by which $v_1$ is mapped to $\vec{x}$.
%We have that $S^\p$ is equal (up to a translation) to $(n+2) S$ and $S_1$ is equal (up to a translation) to $(n+1) S$.
%Consider the image of the vector $(v^\p_1, \vec{x})$ (that is, the vector with the starting point $v^\p_1$ and endpoint $\vec{x}$) under the homothety of $S_1$ to $S + \vec{\alpha}$ and denote the resulting vector by $(x, \vec{y})$.
%Then $\vec{y} \in S + \vec{\alpha}$, $|\vec{x} - v^{\p}_1| = (n+1)|\vec{y} - \vec{x}|$ and thus $|\vec{y} - \vec{x}| < |v^\p - \vec{x}|$.
%Therefore if we consider the homothety with the center $\vec{y}$ and coefficient $(n+2)$ then $\vec{x}$ is mapped into $v^{\p}_1$
%and thus $S + \vec{\alpha}$ is mapped  into $S^\p$.

%Now we can consider the polytope $P^\p$ and the translation $P + \vec{\alpha}$. For this translation we have that
%$v_1$ goes to $\vec{x}$ and thus $\vec{x}$ is a vertex of $P+\vec{\alpha}$.
%Once again the homothety with the center $\vec{y}$ and the coefficient $(n+2)$ sends $\vec{x}$ to $v^\p_1$ and thereby
%$P+\vec{\alpha}$ to $P^{\p}$. Thus $P+\vec{\alpha}$ lies in $P^{\p}$. 
It is only left
to note that the points $\vec{x}$, $\vec{y}$, $v^\p$ lie on the same line in the specified order
and all lie in $P^\p$. Thus by Lemma~\ref{lem:boundary} since $\vec{x} \in \inter{P}^{\p}$ we have  $\vec{y} \in \inter{P}^{\p}$.
\end{proof}

\begin{remark}
We note that Lemma~\ref{lem:discretization} does not hold for $P^{\prime} = (n+1) P$.
The example is very simple, just let $P$ to be a standard simplex, that is a convex hull of points $\{0, \vec{e}_1, \ldots, \vec{e}_n\}$.
Then $P^{\prime} = (n+1) P$ is a convex hull of points $\{\vec{0}, (n+1)\vec{e}_1, \ldots, (n+1)\vec{e}_n\}$.
Let $\vec{x}$ be the center of the polytope $P^{\prime}$, that is $\vec{x} = \vec{e}_1 + \ldots + \vec{e}_n$.
Then for $\vec{x}$ to be a vertex of $P + \vec{\alpha}$ we should have that either $\vec{\alpha} = \vec{e}_1 + \ldots + \vec{e}_n$,
or $\vec{\alpha} = \vec{e}_1 + \ldots + \vec{e}_{i-1} + \vec{e}_{i+1} + \ldots + \vec{e}_n$ for some $i$.
In the first case $\vec{y} = (n+1)(\vec{e}_1 + \ldots + \vec{e}_n)/n$ and in the second case $\vec{y} = (n+1)(\vec{e}_1 + \ldots + \vec{e}_{i-1} + \vec{e}_{i+1} + \ldots + \vec{e}_n)/n$.
In both cases $\vec{y}$ lies on the boundary of $P^{\prime}$: in the first case it is in the convex hull of $\{\vec{e}_1, \ldots, \vec{e}_n\}$ and in the second case it is
in the convex hull of $\{\vec{0}, \vec{e}_1, \ldots, \vec{e}_{i-1}, \vec{e}_{i+1}, \ldots, \vec{e}_n\}$.
\end{remark}

\begin{lemma} \label{lem:touching}
For any point $\vec{x}$ on the bottom of $P_0$ and for any $P_j$ there is $\vec{\alpha}$ such that $P_j + \vec{\alpha}$ touches $P_0$ at $\vec{x}$.
\end{lemma}

\begin{proof}
Let
$$
P = P_1 + \ldots + P_k.
$$
Thus $P_0 = (n+2)P$.

First we show that there is a translation of $P$ touching $P_0$ at $\vec{x}$.

If $\vec{x}$ is a vertex of $P_0$ then just note that there is a translation $P + \vec{\alpha}^\p$ lying inside of $P_0$
and containing $\vec{x}$ (since $P_0$ is a Minkowski sum and $P$ is a summand, $P_0$ can be viewed as a union of translations of $P$). Since $\vec{x}$ is a vertex of $P_0$ it is also a vertex of $P + \vec{\alpha}^\p$.
The homothety $h_{\vec{x}}^{n+2}$ sends $\vec{x}$ as a vertex of $P + \vec{\alpha}^\p$ into $\vec{x}$ as a corresponding vertex of $P_0$ and thus sends $P + \vec{\alpha}^\p$ to $P_0$.
Then by Lemma~\ref{lem:homothety} $P + \vec{\alpha}^\p$ touches $P_0$ at $\vec{x}$.

If $\vec{x}$ is not a vertex of $P_0$
denote the minimal dimension face of $P_0$ containing $\vec{x}$ by $Q_0$.
Clearly, $\vec{x}$ is in the interior of $Q_0$.
Since $P_0 = (n+2)P$ we have that there is a face $Q$ of $P$ such that $Q_0 = (n+2)Q$.
By Lemma~\ref{lem:discretization}
we can find a translation $Q+\vec{\alpha}^\p$ such that $\vec{x}$ is a vertex of $Q+\vec{\alpha}^\p$
and $Q+\vec{\alpha}^\p \subseteq Q_0$.
This lemma also gives us the homothety $h_{\vec{y}}^{n+2}$ which center $\vec{y}$ lies in the interior of $Q_0$.
Now let us consider $P + \vec{\alpha}^\p$
and consider its image under $h_{\vec{y}}^{n+2}$.
The vertex $\vec{x}$ goes under this homothety to the corresponding vertex of $P_0$ and thus $P + \vec{\alpha}^\p$ goes to $P_{0}$.
Note that by Lemma~\ref{lem:homothety} we also get that $P + \vec{\alpha}^\p$ intersects $P_{0}$
only in the faces incident to $\vec{y}$ and thus only in the faces incident to $\vec{x}$.

Now note that $P + \vec{\alpha}^\p$ is the translation of the Minkowski sum of $P_1+ \ldots + P_k$, thus for each of $P_j$
there is a translation $\vec{\alpha}$ such that $P_j + \vec{\alpha}$ is in $P + \vec{\alpha}^\p$ and contains the point $\vec{x}$.
Since this point is a vertex of $P + \vec{\alpha}^\p$ we have that $\vec{x}$ is
a vertex of $P_{j} + \vec{\alpha}$.
Note that $P_j + \vec{\alpha}$ lies inside of $P + \vec{\alpha}^\p$ and thus also can intersect the boundary of $P_{0}$
only in the faces containing $\vec{x}$.
\end{proof}

\begin{corollary} \label{cor:touching}
For any point $\vec{x}$ on the bottom of $P_0$ and for any $f_j$ there is $\vec{\alpha}$ such that $G(\phi_j) + \vec{\alpha}$ touches $P_0$ at $\vec{x}$.
\end{corollary}
\begin{proof}
Note that the set $G(\phi_j) + \vec{\alpha}$ is a subset of $P_j + \vec{\alpha}$, but on the other hand contains
all its vertices. Thus $G(\phi_j) + \vec{\alpha}$ touches $P_0$ at $\vec{x}$.
\end{proof}

\begin{remark}
This section translates to the min-plus case literally.
\end{remark}

\subsection{A facet of $P_{0}$ is singular} \label{sec:singular_facet}

In this subsection we are going to finish the proof of Theorem~\ref{thm:main1}.

For the sake of convenience we will throughout this subsection call an $(n+1)$ dimensional vector $\vec{\alpha}$ (or a point in $\bb{R}^{n+1}$)
integer if its first $n$ coordinates are integers. 

We will not need the following observation in the proof, but it helps to clarify the intuition.
\begin{proposition} \label{cor.p_0}
Consider the bottom $\beta_{P_0}$ of $P_0$, consider the vector $\{a_I\}_I$ corresponding to it,
that is $a_{I} = \beta_{P_0}(I)$ if $\beta_{P_0}(I)$ is defined, and $a_I = \infty$ otherwise.
Consider the tropical polynomial $g = \ta_{I} \left(a_I \tp \vec{x}^I\right)$.
Then for each $f_j$ the polynomial $g$ lies in a tropical ideal generated by $f_j$.
\end{proposition}

\begin{proof}
It is easier to give a proof in geometric terms.
For each integer point $\vec{x}$ on the bottom of $P_0$ consider the translation $G(\phi_j) + \vec{\alpha}_{\vec{x}}$
touching $P_0$ at $\vec{x}$ which exists by Corollary~\ref{cor:touching}.
This translation corresponds to the tropical multiplication of $f_j$ by a monomial.
Then it is easy to see that all the integer points on the bottom of $P_0$ lie in the union of $G(\phi_j) + \vec{\alpha}_{\vec{x}}$ over all $\vec{x}$
and on the other hand all other integer points of this union lie in $P_0$.
The union operation corresponds to the minimum operation (tropical addition) for polynomials.
\end{proof}

\begin{lemma} \label{lem:facet}
Suppose the tropical linear system $M_N \tp \vec{y}$ has a solution $\vec{a}$.
\renewcommand*{\theenumi}{\thetheorem(\roman{enumi})}%
  \renewcommand*{\labelenumi}{(\roman{enumi})}%
\begin{enumerate}
\item \label{lem:facet1} For the case of the tropical semiring $\bb{R}$ there is a face of $P_0$ such that some hyperplane containing it provides a root to the tropical system $F$.

\item \label{lem:facet2} For the case of the tropical semiring $\bb{R}_{\infty}$ if there is $\vec{x} \in \bb{Z}^n$ such that $x \in \Dom \beta_{P_0} \cap \Dom \phi_{\vec{a}}$ then there is a face of $P_0$ such that some hyperplane containing it provides a root to the tropical system $F$.
\end{enumerate}
\end{lemma}

\begin{proof}
Consider the functions $\phi_{\vec{a}}$ and $\beta_{P_0}$.
Since in both cases $\bb{R}$ and $\bb{R}_\infty$ there is $\vec{x} \in \Dom \beta_{P_0}   \cap \Dom \phi_{\vec{a}}$, we have that
there is a singularity point in $\Sing(\phi_{\vec{a}},\beta_{P_0})$.
Further proof works for both cases.

For each point $\vec{x} \in \Sing(\phi_{\vec{a}},\beta_{P_0})$ consider the lowest dimension of the faces of $P_0$
to which the point $(\vec{x}, \beta_{P_0}(\vec{x}))$ belongs and further on denote by $\vec{x}$  the point  in $\Sing(\phi_{\vec{a}},\beta_{P_0})$
which maximizes this minimal dimension.
In simple words, we look for a singularity point in the most general position w.r.t. the polytope $P_0$.
Let us denote the minimal dimension face of $P_0$ containing $(\vec{x}, \beta_{P_0}(\vec{x}))$ by $Q_0$.
Below we show that this face satisfies the formulation of the lemma.

Consider some polynomial $f_j$. By Corollary~\ref{cor:touching} there is a vector $\vec{\alpha}$ such that
$G(\phi_j) + \vec{\alpha}$ touches $P_0$ at $(\vec{x}, \beta_{P_0}(\vec{x}))$.
Denote by $\psi$ the function with the graph $G(\phi_j) + \vec{\alpha}$ (note that $\psi$ corresponds to one of the rows of $M_N$ up to an additive constant).
Then, in particular, we have that $\vec{x} \in \Sing(\beta_{P_0}, \psi)$.
Since we also have $\vec{x} \in \Sing(\phi_{\vec{a}},\beta_{P_0})$ clearly we have $\vec{x} \in \Sing(\phi_{\vec{a}}, \psi)$
(indeed, since $\vec{x}$ by Lemma~\ref{lem:singularity} minimizes functions $\psi - \beta_{P_0}$ and $\beta_{P_0} - \phi_{\vec{a}}$, it also minimizes their sum; note that all three functions are defined at the point $\vec{x}$, so all minimums exist).
However, recall that $\vec{a}$ is a solution to the system $M_N\tp \vec{y}$ and $\psi$ corresponds to one of the rows of $M_N$.
Thus $|\Sing(\phi_{\vec{a}}, \psi)| \geq 2$. But any point minimizing $\psi - \phi_{\vec{a}}$ should also minimize $\psi - \beta_{P_0}$ and $\beta_{P_0} - \phi_{\vec{a}}$ (since $\vec{x}$ does),
hence any point in $\Sing(\phi_{\vec{a}}, \psi)$ should be also in both $\Sing(\phi_{\vec{a}},\beta_{P_0})$ and $\Sing(\beta_{P_0}, \psi)$.
In particular, $|\Sing(\beta_{P_0}, \psi)| \geq |\Sing(\phi_{\vec{a}}, \psi)|$ and thus there is another common point
of $G(\phi_j) + \vec{\alpha}$ and the bottom of $P_0$.

Since $G(\phi_j)+ \vec{\alpha}$ touches $P_0$ at $\vec{x}$ we have that any point in $|\Sing(\beta_{P_0}, \psi)|$ lies in a face of $P_0$ incident to $(\vec{x}, \beta_{P_0}(\vec{x}))$.
If it does not lie in the face $Q_0$, then the minimal dimension face containing
this point has a larger dimension than the dimension of $Q_0$ and we get the contradiction with the maximality property of $(\vec{x}, \beta_{P_0}(\vec{x}))$.
Therefore there are at least two common points of $G(\phi_j)+ \vec{\alpha}$ and $Q_0$. Hence any hyperplane $H$
going through $Q_0$ and not intersecting the interior of $P_0$ is singular to the function corresponding to $G(\phi_j) + \vec{\alpha}$ and thus provides a root to $f_j$.
Since the argument above works for all $f_j$ and $Q_0$ does not depend on $f_j$, we get that $H$ is singular to all $f_1, \ldots, f_k$ and thus defines a root to the system $F$.
\end{proof}

\begin{remark}
In the min-plus case 
the formulation of Lemma~\ref{lem:facet} remains the same.
The proof also almost repeats the proof in the tropical case.
We consider a solution $\vec{a}$ to the min-plus Macaulay system, consider a singular point of $\phi_{\vec{a}}$ and $\beta_{P_0}$ with exactly the same maximization property as before.
We consider the translation of arbitrary $G(\phi_j)$ for arbitrary $j$ as before.
And here we have the only difference to the previous proof. The singularity of $\phi_{\vec{a}}$ and $\psi$
gives us that there is a white singular point and a black singular point
(see Remark~\ref{rk:colors}). Note that this might be the same point with two colors.
Next, the same argument shows that the points of both colors are in $Q_0$, which finishes the proof.
\end{remark}

From Lemma~\ref{lem:facet1} Theorem~\ref{thm:main1} follows immediately. We will use Lemma~\ref{lem:facet2} in the proof of Theorem~\ref{thm:main2} later. The same applies to Theorem~\ref{thm.min-plus}. In Section~\ref{sec:trop_vs_min-plus} we give another proof for the min-plus dual Nullstellensatz with somewhat worse parameters deducing it directly from the tropical dual Nullstellensatz.

\subsection{Examples}
\label{sec:examples}

We provide several examples illustrating why the case of $n>1$ in Theorem~\ref{thm:main1} is substantially harder than the case $n=1$.

\paragraph{Stepped pyramid.} In the case $n=1$ it was actually shown in~\cite{Grigoriev12null} that if we consider any solution $\{a_i\}_{i}$ to the infinite Macaulay system
then if we look onto the large enough $i$'s, then in some natural sense they already form a hyperplane solution,
thereby directly providing a root to the polynomial system. This is not the case already for two variables $n=2$.

To illustrate this consider a tropical polynomial $f$ with a graph
\begin{align*}
G(f) = \{ & (0, 0, 0), (0, 1, 0), (0, 2, 0), (0, 3, 0), \\
& (1, 0, 0), (1, 1, -1), (1, 2, -1), (1, 3, 0), \\
& (2, 0, 0), (2, 1, -1), (2, 2, -1), (2, 3, 0), \\
& (3, 0, 0), (3, 1, 0), (3, 2, 0), (3, 3, 0) \}.
\end{align*}
Its convex hull is an upturned square right pyramidal frustum.

Consider the tropical polynomial system consisting of one polynomial $f$.
For this system we will construct a solution which does not become linear no matter how far away we go from the origin.

It is easier to describe the continuous version of the solution. The discrete solution is defined by the integer points of the continuous solution.

Let $S_k = \{(x,y) | 10(k-1) \leq |x|,|y| \leq 10k\}$ for $k = 1, 2, \ldots$.
For each odd $k$ we let the solution $g \colon \bb{R}^2 \to \bb{R}$ to be constant on $S_k$.
For each even $k$ we divide $S_k$ into $4$ regions by lines $y=x$ and $y=-x$. On the region with
$x \geq |y|$ we let $g(x,y) = x +C$, where $C$ will be chosen later. Analogously for $x \leq -|y|$
we let $g(x,y) = -x +C$, for $y \geq |x|$ let $g(x,y) = y + C$ and for $y \leq -|x|$ let $g(x,y) = -y + C$.
We choose constants in these linear and constant functions in such a way that $g$ is continuous on the whole real plane.
It is not hard to see that the graph of $g$ is singular to the convex hull of $G(f)$.

\paragraph{Stripes.} Now we provide an example demonstrating that solutions of the Macaulay system can behave wildly.
Specifically, we describe almost everywhere ``non-continuous'' solution, that is the solution having arbitrary large gaps in the neighboring points.
For this example also $n=2$ variables suffice.

Consider a polynomial $f$ with
\begin{align*}
G(f) = \{ & (0, 0, 0), (0, 1, -1), (0, 2, 0), \\
& (1, 0, 0), (1, 1, -1), (1, 2, 0)\}.
\end{align*}
The shape of the convex hull of this graph is a prism.

Consider a set of points described by the following function $g \colon \bb{Z}^2 \to \bb{R}$:
$$
\psi(x,y) = \begin{cases} y, & \mbox{if } \lfloor x/2\rfloor\mbox{ is even;} \\ -y, & \mbox{if } \lfloor x/2\rfloor \mbox{ is odd.} \end{cases}
$$
It is not hard to see that the graph of $\psi$ is singular to the convex hull of $G(f)$.
Thus the vector $\{\psi(I)\}_{|I|\leq N}$ is a solution to the Macaulay system corresponding to $f$.
On the other hand note that the gaps in the graph of $\psi$ grow with the growth of $y$.

\subsection{Tropical and Min-plus Dual Nullstellensatz over $\bb{R}_{\infty}$} \label{sec:infinite case}

In this section we prove the following more precise version of Theorem~\ref{thm:main2}.
\begin{theorem} \label{thm:infinity}
Over the semiring $\bb{R}_{\infty}$ a system of tropical polynomials $F = \{f_1, \ldots, f_k\}$ of degree at most $d$ and in $n$ variables has a root iff the Macaulay tropical non-homogeneous linear system with the matrix $M_N$ for
$$
N = 2(n+2)^2 k (4d)^{\min(n,k)+2}
$$
has a solution.
\end{theorem}

\begin{proof}
Suppose we have a system of tropical polynomials $F$ and consider the corresponding
Macaulay tropical non-homogeneous linear system $M_N \tp \vec{y}$. We have already shown one direction: if $F$ has a root then $M_N \tp \vec{y}$ also has a solution.

Suppose in the opposite direction that we have a solution $\vec{a}$ to the non-homogeneous system with the matrix $M_N$.
This means that there is a solution $\vec{a}$ with the finite coordinate corresponding to constant monomial.
If for the enveloping polytope $P_0$ constructed in Subsection~\ref{sec:env_polytope} there is $x \in \bb{Z}^n$ such that $x \in \Dom \beta_{P_0}(\vec{x}) \cap \Dom \phi_{\vec{a}}(\vec{x})$
then we can directly apply Lemma~\ref{lem:facet2}.
But initially we know only that $\phi_{\vec{a}}(\vec{0}) \neq \infty$
and it can be that $\beta_{P_0}(\vec{0}) = \infty$ (and there can be no translation $P_0 + \vec{\alpha}$ of $P_0$ within $\bb{Z}_{+}^n$ such that $\beta_{P_0 + \vec{\alpha}}(\vec{0}) \neq \infty$).
Below we describe how we deal with this problem.

Consider the column of $M_N$ corresponding to the constant monomial.
If it has no finite entry, the non-homogeneous Macaulay system has the infinite solution.
At the same time the system of polynomials also has the infinite solution.
Indeed, note that no polynomial in the system in this case has a finite constant term.
So, this case is simple and further we can assume that the column of $M_N$ corresponding
to the constant monomial has a finite entry.
This means that there is a polynomial in $F$ with a finite constant term.
For simplicity of notation assume that it is $f_1$.

Now based on the system of polynomials $F$ we construct a system of polynomials $F^{\prime}$
such that
\begin{enumerate}
\item each polynomial in $F^\p$ is a (tropical) algebraic combination of polynomials in $F$;
\item each polynomial in $F^\prime$ has a finite constant term;
\item $F^\prime$ has a root iff $F$ also has a root;
\item the number $K$ of polynomials in $F^\p$ is at most  $(n+2)k$ and the maximal degree $d^\p$ of polynomials in $F^\p$ is at most $2 (4d)^{\min(n,k)+2}$.
\end{enumerate}

\begin{claim} \label{cl:F_prime}
If a family of polynomials $F^\p$ with the properties described above exists then Theorem~\ref{thm:infinity} follows.
\end{claim}

\begin{proof}[Proof of the claim]
We only need to show the opposite direction of Theorem~\ref{thm:infinity}. 
Consider the solution $\vec{a}$ to the system $M_N \tp \vec{y}$.
Consider the non-homogeneous Macaulay matrix $M^{\prime}_{N}$ corresponding to $F^{\prime}$. Since all the polynomials in $F^{\prime}$
are tropical algebraic combinations of the polynomials in $F$,
the rows of $M^{\prime}_{N}$ are tropical linear combinations of the rows of $M_N$.
Hence $\vec{a}$ is a solution to $M^{\prime}_{N} \tp \vec{y}$.
Consider the extended Newton polytopes $P_1^\p, P_2^{\prime} \ldots, P_K^{\prime}$ for the polynomials in the system $F^{\prime}$
and consider the enveloping polytope $P_0^{\prime}$ (see Subsection~\ref{sec:env_polytope}).
Note that for each function $f \in F^{\prime}$ we have $\phi_f(\vec{0}) \neq \infty$. Thus the same is true for the corresponding polytopes and
for the enveloping polytope $P_0^\prime$ as well.
Note that we have $N =   2(n+2)^2 k(4d)^{\min(n,k)+2} = (n+2)K  d^\p$.
Therefore Lemma~\ref{lem:facet2} is applicable and
we obtain a root $\vec{b} = (-b_1, \ldots, -b_n) \in \bb{R}^n$ for $F^\prime$.
This in turn implies that $F$ has a root.

\end{proof}

Thus to finish the proof of Theorem~\ref{thm:infinity} it is left to construct the system of polynomials $F^\p$ and ensure its properties.

The idea for the construction of $F^\p$ is to incorporate $f_1$ into all polynomials $f_i$ in $F$ in order to insure that their constant terms are finite.
That is, each polynomial in $F^\p$ will be a (tropical) sum of $f_1$ and an algebraic combination of other polynomials in $F$.
In particular, an extended Newton polytope of a polynomial in $F^\p$ is a convex hull of a union of $P_1$ with some translations of other polytopes among $P_2, \ldots, P_k$.
However, we would like to avoid new roots for the system $F^\p$.
If some face of $P_1$ is a face for all polytopes of the polynomials in $F^\p$ it gives a new root for $F^\p$. 
So we have to construct new polytopes $P^\p$ in such a way that as much as possible of the vertices of $P_1$ lie in the interior of $P^\p$ and thus to reduce the number of faces of $P_1$ in the polytope $P^\p$.
The necessity to avoid new roots in $F^\p$ adds more technical complications. 

We now proceed to the construction of $F^\p$.

Below we will need the following value:
\begin{equation} \label{eq.g}
\Delta = \max_{1 \leq i_1,i_2 \leq k} \max_{\vec{x} \in \Dom \phi_{i_1},\
 \vec{y} \in \Dom \phi_{i_2}}
 | \phi_{i_1}(\vec{x}) - \phi_{i_2}(\vec{y})|.
\end{equation}
Informally, it measures the maximal joint variation of $\phi$-functions for the system $F$.

We also can assume that $\min_{I} \phi_i(I) = 0$ for all $i$
since adding (in the classical sense) a constant to each coefficient of a polynomial does not change singularity.

To construct $F^\prime$ we first for all $i=2, \ldots, k$ and $j = 1, \ldots, n$ define polynomials of the following form:
$$
g_{ij} = (-C) \tp x_j^{\alpha} \tp f_i.
$$
Here the parameters $C$ and $\alpha$ can be fixed in the following way: 
$$
C = 2\Delta (4d)^{2\min(n,k)+2},\ \ \alpha = (4d)^{\min(n,k)+2}.
$$

Next for all $i>1$ we define
\begin{equation} \label{eq:infinity_polynomials}
f_i^{\prime} = f_1 \ta g_{i1} \ta g_{i2} \ta \ldots \ta g_{in}.
\end{equation}
Also for each $i= 2, \ldots, k$ and $j=1,\ldots, n$ we introduce a polynomial
\begin{equation*}
f_{ij}^{\prime} = f_1 \ta g_{i1} \ta g_{i2} \ta \ldots \ta g_{i,j-1} \ta (-1) \tp g_{ij} \ta g_{i,j+1} \ta \ldots \ta g_{in},
\end{equation*}
that is the difference between $f^\prime_i$ and $f_{ij}^\prime$ is that in the latter the coefficient of $x_j^{\alpha} \tp f_i$ is $-C-1$ instead of $-C$.
We let
$$
F^{\prime} = \{f_1\} \cup \{f_{i}^{\prime} \mid i= 2, \ldots, k\} \cup \{f_{ij}^{\prime} \mid i= 2, \ldots, k,\ j=1,\ldots, n\}.
$$
Overall, $F^{\prime}$ consists of $K = (n+1)(k-1) + 1 \leq (n+2)k$ polynomials of degree at most $\alpha + d \leq 2 (4d)^{\min(n,k)+2}$.
For a function $f^\p_i$ we denote its extended Newton polytope by $P^\p_i$ and the function $\phi_{f^\p_i}$ by $\phi^\p_i$.
Analogously, for a function $f^\p_{ij}$ we denote its extended Newton polytope by $P^\p_{ij}$ and the function $\phi_{f^\p_{ij}}$ by $\phi^\p_{ij}$.

The tropical summands of the sum~\eqref{eq:infinity_polynomials} will be called below the \emph{components} of the polynomial $f_i^\prime$.
We specifically distinguish $f_1$-component. All other components are called $f_i$-components.
When we need to distinguish them, the component $g_{ij}$ will be called the $j$-th component of $f^\p_{i}$

All properties of $F^\p$ are clear from the construction except the property $3$.
Moreover, one direction of the property $3$ is simple: since $F^\prime$ consists of algebraic combinations of polynomials of $F$,
any root for $F$ is also a root for $F^\prime$.
Thus it is left to show the following lemma.

\begin{lemma} \label{lem:F_vs_Fprime}
If there is a root to the system $F^\prime$ then there is a root to the system $F$.
\end{lemma}

The proof of this lemma has a geometric intuition, but it is not easy to see the intuition behind the technical details.
So, before proceeding with the proof we would like to explain this intuition in the case of $n=2$ and $k=3$.
After that we provide a formal proof for the general case.

\paragraph{Informal proof for $n=2$ and $k=3$.}

Informally it is convenient to think of constants $C$ and $\alpha$ as of very large parameters to be fixed later. In the formal proof of Lemma~\ref{lem:F_vs_Fprime} we will show that the values of $C$ and $\alpha$ specified above suffice. 

It is instructive to look at the extended Newton polytope $P_2^\p$ corresponding to the function $f^\p_2$. It can be obtained in two steps: we first take a union of the graph of $\phi_1$
and of two copies of the graph of $\phi_2$ translated far away along each of the axes $x_1$ and $x_2$ and far below along the vertical axis and then take an extended Newton polytope of the result.
The idea behind the construction of $f_2^\prime$ is that all the points of the polytope $P_1$ (corresponding to $\phi_1$) except possibly the points on $x_1$-axis and $x_2$-axis are in the interior of the polytope $P_1^{\p}$.

We will explain the presence of the polynomials $f_{21}^\prime, f_{22}^\prime, f_{31}^\prime, f_{32}^\prime$ in $F^\prime$ once we actually need them.

Assume that there is a root $\vec{b} = (- b_1, - b_2)$ to the system $F^\prime$.
Note that since each polynomial in $F^\p$ has finite constant monomial, the value of all polynomials in $F^\p$ is finite on any input. So we can assume that $\vec{b} \in \bb{R}^2$. If it is not the case, just substitute infinite coordinates of $\vec{b}$ by large enough finite numbers. 
Recall, that the root corresponds to the plane $\chi_{\vec{b}}(\vec{x}) = b_1 x_1 + b_2 x_2$ (in $3$-dimensional space)
singular to $\phi_f$ for all polynomials $f \in F^\prime$. 
The first attempt would be to deduce that this hyperplane is also singular to functions $\phi_1, \phi_2, \phi_3$,
corresponding to polynomials $f_1, f_2, f_3$.
We already know that it is singular to $\phi_1$ since $f_1 \in F^\prime$. To show that it is singular to $\phi_2$ and $\phi_3$
we look closer at polynomials $f_2^\prime$ and $f_3^\prime$. Without loss of generality let us consider $f_2^\prime$.

We know that the hyperplane $\chi_{\vec{b}}$ has at least two singular points with $\phi_2^\prime$.
However, if two singular points belong to different components of $f_2^\prime$ it does not give us anything about the singularity of $\chi_{\vec{b}}$ to any $\phi_i$.
Thus, we would like to show that there are two singular points in one of the components of $f_2^\prime$.
Suppose that there is at most one singular point in each component of $f^\p_2$.
We note that if there is at least one singular point in $f_1$-component, then there are two singular points there, since
the hyperplane $\chi_{\vec{b}}$ is singular to $\phi_1$. The case when the hyperplane $\chi_{\vec{b}}$ has only one singular point in one of $f_2$-components
is precisely the case, where we need polynomials $f_{21}^\prime, f_{22}^\prime$.
Indeed, it is not hard to see that in this case one of these polynomials has only one singular point overall, and thus the hyperplane is not singular to $\phi^\p_{21}$ or $\phi^\p_{22}$.

Thus we have that each of the polynomials $f_2^\prime$ and $f_3^\prime$ has at least two singular points in the same component.
If these are $f_2$-component and $f_3$-component respectively then we are done: clearly, the hyperplane is singular to both $\phi_2$ and $\phi_3$.
Thus it is left to consider the case when one of the polynomials (or both) has two singular points in $f_1$-component.

Here we encounter a serious obstacle. 
For example, assume that $\Dom \phi_2$ and $\Dom \phi_3$ do not intersect $\{(t,0)\mid t\in\bb{R}\}$ the set of points on $x_1$-axis. 
Then the hyperplane having two singular points with $\phi_1$ on the $x_1$-axis and decreasing dramatically along the $x_2$-axis
provides a root to $F^\prime$, but not necessarily to $F$.

Thus it is not always true that a root of $F^\prime$ constitutes a root to $F$.
However, in the example described above we can replace $b_2$ by $b_2^\p = \infty$ to obtain a root $\vec{b}^\p$ for $F$.

It turns out that this trick with some additional work is enough to finish the proof.
Indeed, suppose that singular points of the hyperplane $\chi_{\vec{b}}$ and, say, $\phi_2^\prime$ are
in $f_1$-component. Then it is not hard to see that all these singular points lie either on $x_1$-axis, or on $x_2$-axis.
Indeed, for any point $\vec{a} \in \Dom \phi_2^\p$ with both positive coordinates the point $(\vec{a}, \phi_2^\p(\vec{a})) \in \bb{R}^3$ lies in $P_1$ and thus is in the interior of $P_2^\prime$ by the construction of $\phi_2^\p$.
Thus $\vec{a}$ is not a singular point.
If on the other hand, there is a singular point $\vec{a}$ with positive $x_1$-coordinate and another singular point $\vec{a}^\p$ with positive $x_2$-coordinate,
then the point $((\vec{a} + \vec{a^\p})/2,(\phi_1(\vec{a}) + \phi_1(\vec{a^\p}))/2)$ lie in $P_1$ due to its convexity and thus lie in the interior of $P_2^\p$, which contradicts to the singularity of $\vec{a}$ and $\vec{a}^\p$.
Thus this case is also impossible and all singular points lie on one of the axes.

Without loss of generality assume that all the singular points for $f_{2}^\prime$ lie on the $x_1$-axis. Since there are at least two singular points on this axis in $f_1$-component
we have that $b_1$ is not too large and not too small, or more formally, it is bounded in absolute value by $\Delta$ (and thus does not depend on $C$ and $\alpha$).
Since we are allowed to fix $C$ as large as we want, this in particular means that $\Dom(\phi_2)$ does not intersect $x_1$-axis.
Otherwise the singular point of the hyperplane $\chi_{\vec{b}}$ with $\phi_2^\p$ would be in $1$-component and not in $f_1$-component.
Thus to obtain a root of the system $\{f_1, f_2\}$ we can just let $b_2^\p= \infty$ to obtain a new potential root $\vec{b}^\p = (-b_1, -b_2^\p)$.

We would like to stress here that at this point we have shown the theorem for the case $k=2$.
However we need one more observation for the case $k=3$.

Consider the other polynomial $f_{3}^\prime$.
If the domain of $\phi_3$ also does not intersect the $x_1$-axis, then $\vec{b}^\p$ is indeed a root of $f_3^\p$.

Thus we can assume that there is a point $\vec{y}$ in $\Dom(\phi_3)$ on the $x_1$-axis.
Then just like in the case of $\phi_2^\p$ the singular points of $\phi_3^\p$ are not in $f_1$-component.
Thus they are in some $f_3$-components and thus $\chi_{\vec{b}}$ is singular to $\phi_3$ itself.
But we have set $b_2^\p = \infty$ and the singularity might not translate to $\vec{b}^\p$.
This happens if there is only one singular point in $\Sing(\chi_{\vec{b}}, \phi_3)$ on $x_1$-axis. 
So there is another singular point $\vec{z}$ in $\Sing(\chi_{\vec{b}}, \phi_3)$ not on the $x_1$-axis.

Consider both points $\vec{y}, \vec{z} \in \bb{Z}^2$ on the two-dimensional grid. To get from $\vec{y}$ to $\vec{z}$
in this grid we have to make several (at most $d$) steps along $x_1$-axis in positive or negative direction and at least one step in positive direction along $x_2$-axis.
During this path the value of $\phi_3$ and thus of $\chi_{\vec{b}}$ cannot decrease by more than $\Delta$. Indeed,
since $\vec{z}$ is a singular point the difference of the values of $\chi_{\vec{b}}$
is bounded from below by the difference of the value of $\phi_3$ at the same points. 
Since the value of $b_1$ is also bounded in absolute value,
from this we can deduce that $b_2$ is bounded from below by some value depending only on $F$.

Now choosing $C$ large enough we can get a contradiction with the assumption that the singular points of $\phi_{f_2^\prime}$
are in $f_1$-component: both $b_1$ and $b_2$ are not two small and if we place $f_2$-components low enough the singular point will be
in one of these components.

This proof (with some additional technical tricks) can be extended to the general case.

Next we proceed to the formal proof of Lemma~\ref{lem:F_vs_Fprime}.

\paragraph{Proof of Lemma~\ref{lem:F_vs_Fprime}.}
The plan is to consider a root of $F^{\prime}$ and replace some of its coordinates by infinity.
Below we describe how to choose the appropriate set of these coordinates. The construction is rather straightforward:
we only keep the coordinates which we have a reason to keep and the others replace by infinity.

Consider a root $ \vec{b} = (- b_1, \ldots, - b_n) \in \bb{R}^{n}$ of $F^\prime$.
As discussed in Section~\ref{sec:null_prelim} this means that the hyperplane $\chi_{\vec{b}}(\vec{x}) = \sum_i b_i x_i$
is singular to all $\phi_f$ for $f \in F^\prime$.

Note that for each polynomial $f_i^{\prime}$ there are two singularity points in the same component.
Indeed, if this is not the case consider a $j$-component with one singularity point and consider the polynomial $f^\prime_{ij}$.
It has only one singularity point which is a contradiction (the same arguments works for $f_1$-component: we should consider the polynomial $f_1$ in this case).

Below for a set $T \subseteq \bb{R}^n$ and for a set $S \subseteq \{1,\ldots, n\}$ we denote by $\restr{T}{S}$ the set of points $\vec{x} \in T$
such that $x_j = 0$ for all $j \notin S$.

We define the sequence of sets of coordinates in the following iterative way.
First consider the set $\Sing(\chi_{\vec{b}}, \phi_1)$ of singularity points for $\chi_{\vec{b}}$ and $\phi_1$.
We let $j \in S_0$ iff there is $\vec{x} \in \Sing(\chi_{\vec{b}}, \phi_1)$ such that $x_j \neq 0$.
Suppose we have defined $S_l$ by recursion on $l\ge 0$. 
If there is a polynomial $f_i \in F$ such that $\left| \Sing(\chi_{\vec{b}},\phi_i) \right| > \left| \restr{\Sing(\chi_{\vec{b}},\phi_i)}{S_l}\right|$ and
$\restr{\Dom(\phi_i)}{S_l} \neq \emptyset$
then we define $S_{l+1}$ letting $S_l \subseteq S_{l+1}$ and $j \in S_{l+1} \setminus S_{l}$ iff there is $\vec{x} \in \Sing(\chi_{\vec{b}},\phi_i)$ such that $x_j \neq 0$.
If there is no such $f_i$ the process stops.

This procedure results in a sequence $S_0, S_1, \ldots S_r$ and in the corresponding sequence of polynomials
$g_{1}, g_{2}, \ldots, g_{r}$, where for each $l$ we have $g_l \in F$. For the sake of convenience denote $g_{0} = f_1$.
Note that $r \leq k$, since each polynomial from $F$ can appear in the sequence at most once.
Also $r \leq n$, since each $S_l$ is a subset of $\{1,\ldots, n\}$ and each next set is larger than the previous one.
Thus $r \leq \min(n,k)$.

We can pose the following bounds on the coordinates of $\vec{b}$ in $S_r$.
\begin{claim} \label{cl:construction}
For all $l = 0, \ldots, r$ if there is $j \in S_l$ such that $b_{j} \leq  - 2 \Delta (4d)^{l}$ then there is $j^\prime \in S_l$ such that $b_{j^\prime} \geq |b_j|/ (4d)^{l+1}$.
\end{claim}

Informally, if there is a very small $b_j$, then there is rather large $b_{j^\prime}$.

\begin{proof}
We argue by induction on $l$.

For the case of $S_0$ consider the coordinate $j$ with $b_j \leq -2\Delta$ and consider $\vec{x} \in \Sing(\chi_{\vec{b}}, \phi_1)$ such that $x_j \neq 0$ (there is such an $\vec{x}$ by the definition of $S_0$).
Consider $\chi_{\vec{b}}(\vec{x}) - \chi_{\vec{b}}(\vec{0}) = \sum_p x_p b_p$.
Since $\vec{x} \in \Sing(\chi_{\vec{b}},\phi_1)$ by Lemma~\ref{lem:singularity}
$\chi_{\vec{b}}(\vec{x}) - \chi_{\vec{b}}(\vec{0}) \geq \phi_1(\vec{x}) - \phi_1(\vec{0}) \geq -\Delta$.
Note that $x_j > 0$ (and thus $x_j \geq 1$) and for all $p$ we have $x_p \geq 0$, so
$$
\sum_{p \neq j} x_p b_p \geq -x_j b_j - \Delta \geq - b_j - \Delta \geq - b_j /2.
$$
On the other hand note that
$$
\sum_{p \neq j} x_p b_p \leq \max_{p \neq j} b_p \sum_{p \neq j} x_p \leq d \max_{p \neq j} b_p.
$$
Thus there is $j^\prime$ such that $b_{j^\prime} =  \max_{p \neq j} b_p \geq - b_j/2d$.

For the induction step consider $j \in S_l$ such that $b_{j} \leq  - 2 \Delta (4d)^{l}$.
If $j \in S_{l-1}$, we are done by induction hypothesis.
Suppose $j \notin S_{l-1}$.
Consider the polynomial $g_l$.
By its definition we have the following
\begin{enumerate}
\item There is $\vec{y} \in \restr{\Dom(\phi_{g_l})}{S_{l-1}}$. In particular, $y_j = 0$.
\item There is a singular point $\vec{x} \in \Sing(\chi_{\vec{b}},\phi_{g_l})$ such that $x_j \neq 0$.
\end{enumerate}

Consider $\chi_{\vec{b}}(\vec{x}) - \chi_{\vec{b}}(\vec{y}) = \sum_p (x_p - y_p) b_p$.
Due to the singularity of $\vec{x}$ by Lemma~\ref{lem:singularity} we have $\chi_{\vec{b}}(\vec{x}) - \chi_{\vec{b}}(\vec{y}) \geq \phi_{g_l}(\vec{x}) - \phi_{g_l}(\vec{y}) \geq -\Delta$.
Just like in the base of induction we have
$$
\sum_{p \neq j} (x_p - y_p) b_p \geq -(x_j - y_j) b_j - \Delta \geq - b_j - \Delta \geq - b_j /2.
$$
Let us partition the leftmost sum into two parts
$$
\sum_{p \neq j, p \in S_{l-1}} (x_p - y_p) b_p + \sum_{p \neq j, p \notin S_{l-1} } (x_p - y_p) b_p \geq - b_j /2.
$$
Note that in the second sum $(x_p - y_p)$ is nonnegative, because $y_p = 0$.
Since
$$
\sum_{p \neq j} |x_p - y_p| \leq 2d
$$
there is either $p \notin S_{l-1}$ such that $b_p \geq - b_j/4d$, or $p \in S_{l-1}$ such that $|b_p| \geq - b_j/4d$.
In the first case we are done immediately and in the second case we are done by induction hypothesis.
\end{proof}

To obtain the new root we fix all the coordinates of the root not in $S_r$ to $\infty$, that is we let $b_j^{\prime} = -\infty$ if $j \notin S_r$ and $b_j^\p = b_j$ if $j \in S_r$. 
We denote the result by $\vec{b}^\p=(-b_1^\p, \ldots, -b_n^\p) \in \bb{R}^n_{\infty}$.

We claim that $\vec{b}^\prime$ is a root of $F$.

Indeed, suppose there is a polynomial $f_i \in F$ such that there is only one $\vec{z} \in \Sing(\chi_{\vec{b}^{\prime}}, \phi_i)$.
Clearly, $\vec{z} \in \restr{\bb{R}^n}{S_r}$.
Moreover, no other point can be a singular point of the original hyperplane $\chi_{\vec{b}}$ with $\phi_i$.
Indeed, other singular points can be only outside of $\restr{\bb{R}^n}{S_r}$ and if there is at least one,
then following our construction we would have added some more coordinates to $S_r$.
Thus there is only one singular point in $\Sing(\chi_{\vec{b}}, \phi_{i})$ and as a result singular points of $\chi_{\vec{b}}$ with $\phi_{f_i^\prime}$
are in $f_1$-component.
Let $\vec{y} \in \Sing(\chi_{\vec{b}}, \phi_{f_i^\prime})$ be one of these singular points. Note that by the definition of $S_0$ we have $\vec{y} \in \restr{\bb{R}^n}{S_0} \subseteq \restr{\bb{R}^n}{S_r}$.
We are going to get a contradiction.

Let $\Min = \min_{j\in S_r} b_j$ and $\Max = \max_{j \in S_r} b_j$. Consider $j$ and $j^\prime$ such that $b_j = \Min$ and $b_{j^\prime} = \Max$.
Consider $j^\prime$-component of $f_i^\prime$ and let $\vec{x}$ be the translation of $\vec{z}$ in this component,
that is $\vec{x} = \vec{z} + \alpha \cdot \vec{e_{j^\prime}}$.
Since $\vec{z} \in \restr{\bb{R}^n}{S_r}$ and $j^\prime \in S_r$ we have $\vec{x} \in \restr{\bb{R}^n}{S_r}$.

Our goal is to show that $\chi_{\vec{b}}(\vec{x}) - \chi_{\vec{b}}(\vec{y}) > \phi_{f_i^\prime}(\vec{x}) - \phi_{f_i^\prime}(\vec{y})$ which will contradict
Lemma~\ref{lem:singularity} since $\vec{y} \in \Sing(\chi_{\vec{b}}, \phi_{f_i^\prime})$.

Note that
$$
\phi_{f_i^\prime}(\vec{x}) - \phi_{f_i^\prime}(\vec{y}) = \phi_{f_i^\prime}(\vec{z} + \alpha \cdot \vec{e_{j^\prime}}) - \phi_1(\vec{y}) =
(\phi_i(\vec{z}) - C) - \phi_1(\vec{y})
 \leq  - C + \Delta,
$$
where the second equality follows from the definition of $f_{i}^\prime$~\eqref{eq:infinity_polynomials} and the last inequality follows from the definition of $\Delta$~\eqref{eq.g}.
Thus it is enough to show that $\chi_{\vec{b}}(\vec{x}) - \chi_{\vec{b}}(\vec{y}) > -C +\Delta$.

Note now that
$$
x_{j^\prime} - y_{j^\prime} = x_{j^\prime} - z_{j^\prime} + z_{j^\prime} - y_{j^\prime} \geq \alpha - 2d
$$
and
$$
 \sum_{p \neq j^\prime} \left| x_p - y_p \right|  = \sum_{p \neq j^\prime} \left| z_p - y_p \right| \leq 2d.
$$

Consider the sum
\begin{align*}
& \chi_{\vec{b}}(\vec{x}) - \chi_{\vec{b}}(\vec{y}) = &\\
&  \sum_{p \neq j^\prime, x_p - y_p> 0} (x_p - y_p) b_p + \sum_{p \neq j^\prime, x_p - y_p \leq 0} (x_p - y_p) b_p + (x_{j^\prime} - y_{j^\prime}) b_{j^\prime} \geq &\\
& \Min \cdot \sum_{p \neq j^\prime, x_p - y_p> 0} (x_p - y_p) +  \Max \cdot \sum_{p \neq j^\prime, x_p - y_p \leq 0} (x_p - y_p) + \Max \cdot (x_{j^\prime} - y_{j^\prime}).&
\end{align*}
If $\Max \leq 0$, then by Claim~\ref{cl:construction} we have $\Min \geq - 2\Delta (4d)^r$ and thus $\Max \geq - 2\Delta (4d)^r$. We have
\begin{align*}
& \chi_{\vec{b}}(\vec{x}) - \chi_{\vec{b}}(\vec{y}) \geq &\\
& - 2d 2\Delta (4d)^r + 0 - 2\Delta (4d)^r (\alpha - 2d) = - \alpha 2\Delta (4d)^r.&
\end{align*}
If $\Max \geq 0$ we have
\begin{align*}
& \chi_{\vec{b}}(\vec{x}) - \chi_{\vec{b}}(\vec{y}) \geq &\\
& \Min \cdot \sum_{p \neq j^\prime, x_p - y_p> 0} (x_p - y_p) -  \Max \cdot 2d + \Max \cdot (\alpha - 2d) =&\\
& \Min \cdot \sum_{p \neq j^\prime, x_p - y_p> 0} (x_p - y_p) +  \Max \cdot (\alpha - 4d) .&
\end{align*}
If $\Min \geq - 2\Delta (4d)^r$ this sum is greater than $-2\Delta 2d (4d)^r$.

If $\Min \leq- 2\Delta (4d)^r$ then by Claim~\ref{cl:construction} $\Max \geq - \Min/(4d)^{r+1}$ and we have
\begin{align*}
& \chi_{\vec{b}}(\vec{x}) - \chi_{\vec{b}}(\vec{y}) \geq &\\
& 2d \Min  -  (\alpha - 4d) \Min/(4d)^{r+1} \geq 0.&
\end{align*}
In all these cases $\chi_{\vec{b}}(\vec{x}) - \chi_{\vec{b}}(\vec{y})$ is greater than $-C + \Delta$ and we have a contradiction with the singularity of $\vec{y}$.

This finishes the proof of Lemma~\ref{lem:F_vs_Fprime}, Theorem~\ref{thm:infinity} and thus Theorem~\ref{thm:main2}.

\end{proof}

\begin{remark}
For the min-plus case the construction of the system $F^\p$ and its properties translate with obvious changes only (a min-plus polynomial is a pair of tropical polynomials; there is a pair of matrices instead of one Macaulay matrix). The proof of Claim~\ref{cl:F_prime} also translates with obvious changes only. We do not attempt to translate informal proof since this does not add to the intuition. The formal proof is given in terms of $\phi_i$ and $\phi_i^\p$ function, so it also translates easily. The only difference is that instead of property ``to have only one singular point'' we use the property to ``to have singular points of only one color'' (see Remark~\ref{rk:colors}). The rest of the proof remains the same.
In Section~\ref{sec:trop_vs_min-plus} we give another proof for the min-plus dual Nullstellensatz with somewhat worse parameters deducing it directly from the tropical dual Nullstellensatz.
\end{remark}

\subsection{Lower Bounds} \label{sec:lower_bounds}

In this subsection we provide examples showing that our bounds on $N$ in Theorem~\ref{thm:main} are not far from being optimal.
At the same time we provide the similar lower bounds for Theorem~\ref{thm.min-plus}. We will translate these lower bounds to Theorems~\ref{thm.min-plus-primary} and~\ref{thm.tropical-primary} in Section~\ref{sec:primary_nullstellensatz}. 

\paragraph{Lower bound for Theorem~\ref{thm:main1}}
First we show a lower bound for the case of $\bb{R}$.
Namely for any $d\geq 2$ we provide a family $F$ of $n+1$ polynomials in $n$ variables and of degree at most $d$ such that $F$ has no root,
but the corresponding Macaulay system $M_{(d-1)(n-1)}\tp \vec{y}$ has a solution.

The construction is an adaptation of the known example for a lower bound for classical Nullstellensatz
due to Lazard, Mora and Philippon (unpublished, see~\cite{Brownawell87,GrigorievV01}).

Consider the following set $F$ of tropical polynomials
\begin{align*}
f_1 &= 0 \ta 0 \tp x_1,\\
f_{i+1} &= 0 \tp x_i^{\tp d} \ta 0 \tp x_{i+1},\ 1\leq i \leq n-1\\
f_{n+1} &= 0 \ta 1 \tp x_n.
\end{align*}
It is not hard to see that this system has no roots.
Indeed, if there is a root, then from $f_1$ we get that $x_1 = 0$,
then from $f_2$ we get that $x_2=0$ etc., finally from $f_n$ we conclude that $x_n=0$.
However from $f_{n+1}$ we have that $x_{n}=-1$ which is a contradiction.

Thus it remains to show that the Macaulay tropical system with the matrix $M_{(d-1)(n-1)}$ corresponding to the system $F$ has a solution.

Recall that the columns of $M_{(d-1)(n-1)}$ correspond to monomials.
We associate an undirected graph $G$ to the matrix $M_{(d-1)(n-1)}$ in a natural way.
The vertices of $G$ are monomials in variables $x_1, \ldots, x_n$ of degree at most $(d-1)(n-1)$
(or, which is the same, the columns of $M_{(d-1)(n-1)}$). We connect two monomials by an edge if they occur
in the same polynomial of the form $\vec{x}^I \tp f_{i}$ for $i \in\{1, \ldots, n\}$. Or, to state it the other way,
we connect two monomials if there is a row of $M_{(d-1)(n-1)}$ not corresponding to a polynomial $f_{n+1}$
and such that the entries in the columns corresponding to these monomials are both finite in this row.

We assign the weight $w(m)$ to a tropical monomial $m$ in the following way.
First, $w(x_i) = d^{i-1}$ for all $i = 1, \ldots, n$. Second, for all monomials $m_1$ and $m_2$ we let $w(m_1 \tp m_2) = w(m_1) + w(m_2)$.
That is, if $m = x_{1}^{\tp a_1} \tp \ldots \tp x_{n}^{\tp a_n}$ then $w(m) = a_1 + a_2 d + \ldots a_n d^{n-1}$.

It turns out that the following lemma holds.

\begin{lemma} \label{lem:connectivity}
If for two monomials $m_1$ and $m_2$ we have $w(m_1) \geq kd^{n-1}$ and $w(m_2) < kd^{n-1}$ for some integer $k$,
then $m_1$ and $m_2$ are not connected in $G$.
\end{lemma}

\begin{proof}

Note that if two monomials are connected by an edge corresponding to one of the polynomials $f_2, \ldots, f_n$,
then their weights coincide. If they are connected by an edge corresponding to $f_1$, then their weights differ by $1$.

Note that for arbitrary $k$ any monomial of weight $kd^{n-1} - 1$ has the degree at least $(k - 1) + (d-1)(n-1)$.
Indeed, consider such a monomial $m = x_{1}^{\tp a_1} \tp \ldots \tp x_{n}^{\tp a_n}$ of the minimal degree.
If there is $i = 1, \ldots, n-1$ such that $a_i \geq d$, then we can replace $a_i$ by $a_i - d$ and $a_{i+1}$ by $a_{i+1} + 1$
and obtain another monomial of the same weight but with a smaller degree.
Thus $(a_1, \ldots, a_{n-1})$ corresponds to $d$-ary representation of the residue of $kd^{n-1}-1$ modulo $d^{n-1}$.
So for all $i=1, \ldots, n-1$ we have $a_i = d-1$ and thus $a_n = k-1$.

Due to the restriction on the degree in the graph $G$ there is only one monomial of weight $d^{n-1} - 1$
and no monomials of weight $kd^{n-1} - 1$ for $k>1$. Moreover, the unique monomial of weight $d^{n-1} -1$ has the maximal degree $(d-1)(n-1)$
and thus is not connected to a monomial of a higher weight by an edge.

From all this the lemma follows. Indeed, if monomials $m_1$ and $m_2$ are connected, then on the path between them there is
an edge connecting monomials of weights $kd^{n-1}-1$ and $kd^{n-1}$. However, as we have shown, this is impossible.

\end{proof}

Now we are ready to provide a solution to the Macaulay system $M_{(d-1)(n-1)}\tp \vec{y}$. For a monomial of weight
$kd^{n-1} + s$, where $s < d^{n-1}$, set the corresponding variable of $\vec{y}$ to $k$.
Note that due to Lemma~\ref{lem:connectivity} if two monomials are connected, then the values of the corresponding variables are the same.
Thus the constructed $\vec{y}$ satisfies equations of $M_{(d-1)(n-1)}\tp \vec{y}$ corresponding to polynomials $f_1, \ldots, f_n$.
The rows corresponding to $f_{n+1}$ are satisfied since the weights of monomials differ by precisely $d^{n-1}$.

\begin{remark}
For the case of min-plus polynomials a straightforward adaptation works.
Indeed, since in each polynomial there are only two monomials,
the only way to satisfy them is to make their values equal.
Thus it is enough to consider a system of min-plus polynomials $F$
\begin{align*}
&(0,  0 \tp x_1),\\
&(0 \tp x_i^{\tp d},  0 \tp x_{i+1}),\ 1\leq i \leq n-1,\\
&(0,  1 \tp x_n).
\end{align*}
\end{remark}

\paragraph{Lower bound for Theorem~\ref{thm:main2}}

Now we show a lower bound for the case of $\bb{R}_{\infty}$.

Consider the following system $F$ of tropical polynomials in variables $x_1, \ldots, x_n, y$.
\begin{align*}
f_1 &= 0 \tp x_1 \tp y \ta 0,\\
f_{i+1} &=  0\tp x_i^{\tp d} \ta 0 \tp x_{i+1}, \text{for } i=1, \ldots, n-1,\\
f_{n+1} &= 0 \tp x_{n-1}^{\tp d} \ta 1 \tp x_n.
\end{align*}

This system clearly has no roots. Indeed, we can consecutively show that all coordinates of a root should
be finite and then the polynomials $f_n$ and $f_{n+1}$ give a contradiction.

Now consider the Macaulay non-homogeneous system with a matrix $M_{d^{n-1}-1}$.
We are going to construct a solution for it.
For a tropical monomial $x_{1}^{a_1} \ldots x_{n}^{a_n} y^{b}$ let its weight be
$$
a_1 + d a_2 + d^2 a_3 + \ldots + d^{n-1} a_n
$$
Note that the degree in $y$ is not counted.
Consider monomials whose $y$-degree coincides with their weight and let the corresponding coordinates of $\vec{z}$ be equal to $0$.
For all other monomials let the corresponding coordinates of $\vec{z}$ to be equal to $\infty$.
We show that indeed, this provides a solution to $M_{d^{n-1}-1} \tp \vec{z}$.
Consider the graph on the coordinates of solution in which two coordinates are connected if the corresponding
monomials appear in the same row of Macaulay matrix $M_{d^{n-1}-1}$.
It is not hard to see that all the monomials on which our solution is finite constitute a connected component of the graph, containing zero coordinate.
Moreover, due to the constraint on the size of the matrix, no monomials in this component contain $x_n$ variable.
Thus, all the rows of the Macaulay matrix are satisfied.

\begin{remark}
For the min-plus case note that the same observation as in the case of $\bb{R}$ works.
Just consider the system of min-plus polynomials
\begin{align*}
&(0 \tp x_1 \tp y,  0),\\
&(0\tp x_i^{\tp d},  0\tp x_{i+1}), \text{for } i=1, \ldots, n-1,\\
&(0\tp x_{n-1}^{\tp d},  1 \tp x_n).
\end{align*}
\end{remark}

\section{Tropical polynomials vs. Min-plus polynomials} \label{sec:trop_vs_min-plus}

In this section we show that there is a tight connection between
the sets of roots of systems of min-plus polynomials and of tropical polynomials.
We will later use this connection to obtain the min-plus dual Nullstellensatz.

A connection in one direction is simple.

\begin{lemma}
Over both $\bb{R}$ and $\bb{R}_{\infty}$ for any given tropical polynomial system we can construct a system of min-plus polynomials
over the same set of variables, with the same set of roots and of the same degree.
\end{lemma}

\begin{proof}
Let $T$ be some tropical polynomial system over $\bb{R}$.
For each polynomial $f \in T$ we construct a min-plus polynomial system
over the same set of variables which is equivalent to $f$.

For this let
\begin{equation} \label{eq.row}
f = \min\{l_1, l_2, \ldots, l_m\},
\end{equation}
where $l_i$'s are tropical monomials.

It is easy to see that the minimum in~\eqref{eq.row} is attained at
least twice iff for all $i = 1, \ldots, m$
it is true that
\begin{align*}
& \min\{l_1, ..., l_{i-1}, l_{i}, l_{i+1}, ..., l_{m}\}
= \\
& \min\{l_1, ..., l_{i-1}, l_{i+1}, ..., l_{m}\}.
\end{align*}
These equations are min-plus polynomial equations and thus we have that any tropical polynomial is equivalent
to a system of min-plus polynomials. To get a min-plus system equivalent to the tropical system
we just unite min-plus systems for all polynomials of $T$.

Exactly the same analysis works for the case of $\bb{R}_{\infty}$.
\end{proof}

In the opposite direction we do not have such a tight connection,
but the connection we show below still preserves many properties.

We first for a given min-plus polynomial system $A$ construct a corresponding tropical polynomial
system $T$ and then prove a relation between $A$ and $T$.

Let us denote variables of $A$ by $(x_{1}, \ldots, x_{n})$.
The tropical polynomial system $T$ for each variable $x_{i}$ of $A$
will have two variables $x_{i}$ and $x_{i}^{\prime}$,
thus the set of variables of $T$ will be $(x_{1}, \ldots, x_{n}, x^{\prime}_{1}, \ldots, x^{\prime}_{n})$.

Polynomial system $T$ consists of the following polynomials.
\begin{enumerate}
  \item For each $i=1,\ldots,n$ we add to $T$ a polynomial
  $$
  x_i \ta x_i^{\prime}.
  $$
  \item Let $\left(\min_j m_j(\vec{x}), \min_p l_p(\vec{x})\right)$ be an arbitrary polynomial of $A$.
  For each $p$ we add to $T$ a tropical polynomial
  \begin{equation} \label{eq:reduction1}
  \min\left(m_1(\vec{x}), m_1(\vec{x}^{\prime}), \ldots, m_k(\vec{x}), m_k(\vec{x}^{\prime}), l_p(\vec{x})\right).
  \end{equation}
  For each $j$ we add to $T$ a tropical polynomial
  \begin{equation} \label{eq:reduction2}
  \min\left(l_1(\vec{x}), l_1(\vec{x}^{\prime}), \ldots, l_k(\vec{x}), l_k(\vec{x}^{\prime}), mß_j(\vec{x})\right).
  \end{equation}
  We denote monomials $m_1, \ldots, m_k$ in these polynomials by $m$-monomials. We denote monomials $l_1, \ldots, l_k$ by $l$-monomials. 
\end{enumerate}

This completes the construction of $T$. Note that the maximal degree of polynomials in $T$
is equal to the maximal degree of polynomials in $A$.
Now we are ready to show how $A$ and $T$ are related.

\begin{lemma} \label{lem:min-plus_to_trop}
There is an injective (classical) linear transformation $H \colon \bb{R}_{\infty}^n \to \bb{R}_{\infty}^{2n}$ such that
all the roots of $T$ lie in $Im(H)$
and the image of the set of roots of $A$ coincides with the set of roots of $T$.
The same is true for the semiring $\bb{R}$.
\end{lemma}

\begin{proof}
Let $H(\vec{a}) = (\vec{a},\vec{a})$ for all $\vec{a}$.
Clearly, $H$ is an injective linear transformation.

Note that the polynomial of $T$ of the first type $x_i \ta x_i^{\prime}$ is satisfied iff $x_i = x_i^{\prime}$.
Thus all the roots of $T$ lie in the image of $H$.

If there is a root $\vec{a}$ to $A$ then it is easy to see that its image $(\vec{a},\vec{a})$ under $H$ satisfies
all the polynomials of the form~\eqref{eq:reduction1} and~\eqref{eq:reduction2} in $T$.
Indeed, since $\min_j m_j(\vec{a}) = \min_p l_p(\vec{a})$, then there is $j$ such that $m_j(\vec{a}) = \min_p l_p(\vec{a})$.
Then the minimum in the corresponding tropical polynomials~\eqref{eq:reduction1} will be attained in monomials $m_j(\vec{x})$ and $m_j(\vec{x}^\prime)$.
The symmetric argument works for tropical polynomials~\eqref{eq:reduction2}.

If there is a root of $T$, we already noted that it has the form $(\vec{a},\vec{a})$.
Then it is not hard to see that for each min-plus polynomial of $A$ we have $\min_j m_j(\vec{a}) = \min_p l_p(\vec{a})$.
Indeed, since corresponding tropical polynomials~\eqref{eq:reduction1} are satisfied, we have that $\min_j m_j(\vec{a}) \leq \min_p l_p(\vec{a})$.
On the other hand, tropical polynomials~\eqref{eq:reduction2} guarantee that 
$$\min_j m_j(\vec{a}) \geq \min_p l_p(\vec{a}).$$

The proof works over both semirings $\bb{R}$ and $\bb{R}_\infty$.

\end{proof}

The sets of roots of tropical (or min-plus) systems of polynomials are called \emph{tropical (respectively, min-plus) prevarieties}.
In particular, it follows that the classes of tropical prevarieties and min-plus prevarieties are topologically equivalent.

To illustrate possible applications of this connection we deduce another proof of the min-plus dual Nullstellensatz with somewhat worse values of $N$ directly from  the tropical dual Nullstellensatz.

We present the proof for the semiring $\bb{R}$. Exactly the same proof works also for $\bb{R}_\infty$.

As usually one direction is simple, that is
if a system $F$ has a root $\vec{a}$, then the Macaulay min-plus linear system $\Ml_N \tp \vec{y} = \Mr_N \tp \vec{y}$ for any $N$ also has a solution:
just let each coordinate $y_I$ of $\vec{y}$ to be equal to $\vec{a}^I$.

For the opposite direction, suppose the system $\Ml_N \tp \vec{y} = \Mr_N \tp \vec{y}$  for $N = (2n+2)(d_1 + \ldots + d_k)$ has a solution $\vec{a}$.
For the min-plus polynomial system $F$ consider the corresponding tropical polynomial system $T$ from the previous section.
Let us denote by $M_N^{\p}$ its Macaulay matrix. We will show that the tropical linear system $M_N^{\p} \tp \vec{z}$ has a solution.
From this by Theorem~\ref{thm:main} it follows immediately that $T$ has a root (note that the number of variables is $2n$, hence the change in the value of $N$), and by Lemma~\ref{lem:min-plus_to_trop}
we have that $F$ has a root.

Thus it is left to construct a solution to the tropical system $M_N^{\p} \tp \vec{z}$ based on the solution to the min-plus system $\Ml_N \tp \vec{y} = \Mr_N \tp \vec{y}$.
The construction is straightforward: for each monomial $m(\vec{x},\vec{x}^\prime)$ we partition it into two parts $m_1(\vec{x}) \tp m_2(\vec{x}^\prime)$ containing variables
$\vec{x}$ and $\vec{x}^\prime$ respectively and we let the variable of $\vec{z}$ corresponding to $m(\vec{x},\vec{x}^\prime)$ to be equal to the variable of $\vec{y}$ corresponding to the monomial $m_1(\vec{x})\tp m_2(\vec{x})$.

Now we have to check that all the rows of the system $M_N^{\p} \tp \vec{z}$ are satisfied.
This is obviously true for the rows corresponding to the polynomials $\vec{x}^I \tp \vec{x}^{\prime I^\prime} \tp (x_i \ta x_i^\prime)$,
since we clearly assign the same value to the variables corresponding to both monomials.
For the polynomials of the form~\eqref{eq:reduction1} we consider the corresponding equation in $\Ml_N \tp \vec{y} = \Mr_N \tp \vec{y}$.
Since the minimum in them is attained in $m$-monomials in the corresponding row of $M_N^{\p} \tp \vec{z}$ the minimum will be attained in
two corresponding $m$-monomials. The same works for the polynomials of the form~\eqref{eq:reduction2}.

\section{Tropical and Min-plus Primary Nullstellens\"atze} \label{sec:primary_nullstellensatz}

Now we will deduce primary forms of the tropical and min-plus Nullstellens\"atze.
We start with the min-plus primary Nullstallensatz.

\begin{proof}[Proof of Theorem~\ref{thm.min-plus-primary}]
We will use the min-plus linear duality for the proof of this theorem.
We start with the case of the semiring $\bb{R}$.

By Theorem~\ref{thm.min-plus1} the system of polynomials $F$ has no roots over $\bb{R}$ iff the corresponding Macaulay linear system
$$
\Ml_N \tp \vec{y} = \Mr_N \tp \vec{y}
$$
has no finite solution.
This system is equivalent to the system of min-plus inequalities
$$
\left(
  \begin{array}{c}
    \Ml_N \\
    \Mr_N \\
  \end{array}
\right) \tp \vec{x} \leq
\left(
  \begin{array}{c}
    \Mr_N \\
    \Ml_N \\
  \end{array}
\right) \tp \vec{x}.
$$

By Corollary~\ref{cor:min-plus_linear_duality} this system has no finite solution iff the dual system
$$
\left(
  \begin{array}{ccc}
    \Mr_N^{T} & \Ml_N^{T}\\
  \end{array}
\right) \tp
\left(
  \begin{array}{c}
 \vec{y} \\
    \vec{z} \\
  \end{array}
\right) <
\left(
  \begin{array}{cc}
    \Ml_N^{T} & \Mr_N^{T} \\
  \end{array}
\right) \tp
\left(
  \begin{array}{c}
 \vec{y} \\
    \vec{z} \\
  \end{array}
\right)
$$
has a solution $(\vec{y},\vec{z}) \neq (\infty,\ldots, \infty)$ (recall that we allow for both sides to be infinite in some rows).

This system can be interpreted back in terms of polynomials.
Indeed, note that now the columns of the matrices correspond to the equations of $F$ multiplied by some $\vec{x}^J$ and the rows correspond to
some monomials $\vec{x}^I$. Thus a solution to the system corresponds to the sum of equations of $F$ multiplied by some monomials, such that
each coefficient of the sum on the left-hand side is smaller than the respective coefficient of the sum on the right-hand side.
The fact that we allow both sides to be infinite in some row corresponds to the fact that some monomials might
not occur in the sum.
The fact that we allow infinite coordinates in the solution corresponds to the fact that we do not have to use all the polynomials of the form $(\vec{x}^I \tp f_j, \vec{x}^I \tp g_j)$
in the algebraic combination.

The proof of the second part of the theorem is almost the same.
The only difference is that this time we should use non-homogeneous Macaulay system, which by an application of Lemma~\ref{lem:min-plus_linear_duality} results in a linear combination
of polynomials with a finite constant term.
\end{proof}

Now we proceed to the tropical primary Nullstellensatz.

\begin{proof}[Proof of Theorem~\ref{thm.tropical-primary}]
By Theorem~\ref{thm:main1} the system of polynomials $F$ has no roots over $\bb{R}$ iff the corresponding Macaulay system
$$
M_N \tp \vec{y}
$$
has no finite solution.

By Corollary~\ref{cor:tropical_linear_duality} this is equivalent to the fact that there is $\vec{z} \neq (\infty, \ldots, \infty)$ in $\bb{R}_{\infty}^{n}$
such that in each row of
$$
M_{N}^{T} \tp \vec{z}
$$
the minimum is attained only once or is equal to $\infty$ and for each two rows the minimums are in different columns.
Recall, that each column in $M^T_N$ corresponds to a polynomial $\vec{x}^J \tp f_j$ and the rows correspond to the monomials $\vec{x}^I$
in these polynomials. Thus $\vec{z}$ corresponds to the algebraic combination of polynomials of $F$ and the properties of $\vec{z}$
described above are equivalent to the singularity of the corresponding algebraic combination.

The proof of the $\bb{R}_\infty$ case is completely analogous.
\end{proof}

The lower bounds on $N$ in Theorems~\ref{thm.min-plus-primary} and~\ref{thm.tropical-primary} can be proved along the same lines as the proofs above by considering polynomial systems constructed in Subsection~\ref{sec:lower_bounds}. It was shown there that for these polynomial systems the Macaulay linear systems of certain size have solutions. Using tropical and min-plus linear duality and interpreting the results in terms of polynomials just like in the proofs above we can show that these polynomial systems provide lower bounds for the tropical and min-plus primal Nullstellens{\"a}tze.

\section{Linear duality in min-plus algebra} \label{sec:linear_duality}

\subsection{Min-plus linear duality}

In this subsection we prove Lemma~\ref{lem:min-plus_linear_duality} on the duality for min-plus linear systems.

The proof of this lemma is based on the interpretation of min-plus linear systems as mean payoff games.
Namely, given a min-plus linear system we construct a mean payoff game $G$ such that a solution to the system corresponds to a winning strategy for one of the players.
This connection between min-plus linear systems and mean payoff games was established in~\cite{akian12mean_payoff}.
We present the details here for the sake of completeness.

Given two matrices $A,B \in \bb{R}^{m\times n}_{\infty}$ the corresponding mean payoff game $G$ can be described as follows.
Consider a directed bipartite graph which vertices on the left side are $r_1, \ldots, r_m$
and vertices on the right side are $c_1, \ldots, c_n$. Left-side vertices correspond to the rows of matrices $A$ and $B$
and right-hand side vertices correspond to the columns of the matrices.
From each vertex $r_i$ there is an edge to a vertex $c_j$ labeled by $-a_{ij}$.
From each vertex $c_j$ there is an edge to a vertex $r_i$ labeled by $b_{ij}$.
We denote the label of an edge $(v,u)$ by $w(v,u)$.
Thus $w(r_i, c_j) = -a_{ij}$ and $w(c_j, r_i) = b_{ij}$.
There are two players which we call the row-player and the column-player and who in turns are moving
a token over the vertices of the graph.
In the beginning of the game the token is placed to some fixed vertex. On each turn one
of the two players moves the token to some other node of the graph. Each turn of the
game is organized as follows. If the token is currently in some node $r_i$ then the column-player can move it to any
node $c_j$ (the column-player chooses a column). If, on the other hand, the token is in some node $c_j$ then the row-player can move the token to any node
$r_i$ (the row-player chooses a row). The game is infinite and the process of the game can be described by a sequence of nodes
$v_{0}, v_{1}, v_{2}, \ldots$ which the token visits. 
The column-player wins the game if
\begin{equation} \label{eq.mpg}
\liminf_{t \to \infty} \frac 1t \sum_{i=1}^{t} w(v_{i-1}, v_{i}) > 0.
\end{equation}
If this limit is negative then the row-player wins. If the limit is zero we have a draw.
If some entries of matrices $A, B$ are infinite we assume that there are no corresponding edges in the graph.
Alternatively, we can assume that there are edges labeled by $\infty$ and the player following such an edge
losses immediately.

The process of the game can be viewed in the following way. After each move of the column-player
he receives from the row-player some amount $-a_{ij}$ and after each move of the row-player
he receives from the column-player some amount $-b_{ij}$. The goal of both players is to maximize their amount.
If one of them can play in such a way that his amount grows to infinity as the game proceeds, then he wins.
And if the amounts of the players always stay between some limits, then the result of the game is the draw.

Note that if all the entries of the matrices are finite the game has a complete bipartite graph and it is easy to see that
this property implies that the winner of the game does not depend on the starting position.
The situation is different in the case of matrices with entries from $\bb{R}_{\infty}$.

For the constructed game $G$ the following property holds. It is implicit in~\cite{akian12mean_payoff}.
\begin{lemma} \label{lem:non-losing}
There is a finite solution to $A \tp \vec{x} \leq B \tp \vec{x}$ iff the column-player has a non-losing strategy starting from any position.

There is a solution $\vec{x} \neq (\infty, \ldots, \infty)$ to $A \tp \vec{x} \leq B \tp \vec{x}$
iff the column-player has a non-losing strategy starting from some position.

There is a solution to $A \tp \vec{x} \leq B \tp \vec{x}$
with a finite coordinate $x_i$ iff the column-player has a non-losing strategy starting from the position $c_i$.

There is a solution to $A \tp \vec{x} \leq B \tp \vec{x}$
such that the $j$-th coordinate of $A \tp \vec{x}$ is finite iff the column-player has a non-losing strategy starting from the position $r_j$.
\end{lemma}

\begin{proof}
We always can add the same number to all coordinates of the solution.
In particular we have that there is a solution $\vec{x} \neq (\infty,\ldots,\infty)$ to $A \tp \vec{x} \leq B \tp \vec{x}$ iff there is a solution
such that all $x_j \geq 0$ and $\min_{j}{x_j} = 0$.

We are going to show that the existence of such a solution is equivalent to the existence of non-losing
strategy for the column-player.
The proof is very intuitive, but to make the intuition clear we have to explain what does $\vec{x}$ mean in terms of the game.
To do this assume that the column-player has a non-losing strategy starting from some position. We know that if the player follows the strategy,
then his amount does not decrease to $-\infty$. But it might become negative at some moments of the game.
For an arbitrary vertex $c_j$ let us denote by $x^{\p}_j$ the minimal amount such that if the game starts in $c_j$
and the column-player has $x^{\p}_j$ in the beginning then he can never go below zero.
If in some position $c_j$ the column-player has no winning strategy we naturally set $x^\p_j = \infty$.
It turns out that the vector $\vec{x}^\p$ is a solution to the min-plus linear system.

Indeed, suppose that the column-player has a non-losing strategy and consider $\vec{x}^{\p}$ corresponding to it.
Assume that we are in position $c_j$. Then for each move of the row-player $(j,i)$ there is a move of the column-player $(i,k)$
such that the remaining amount of the column-player after these two moves is at least $x^\p_k$ (so he does not go below his budget in the future).
That is for each $i$ and $j$ there is a $k$ such that
$
x^\p_j + b_{ij} - a_{ik}  \geq x^\p_k
$
or
$$
x^\p_k + a_{ik} \leq x^\p_j  + b_{ij}.
$$
And this precisely means that $A \tp \vec{x}^\p \leq B \tp \vec{x}^\p$.

Now, suppose that there is a solution $\vec{x}$ to the min-plus linear system.
Let us give the column-player the amount $x_j$ if the game starts in $c_j$.
Then reversing the argument we have that
for each $i$ and $j$ there is $k$ such that
$
x_j + b_{ij} - a_{ik} \geq x_k
$.
And this means that for each position $c_j$ and for each move $(j,i)$ of the row-player
there is a move $(i,k)$ of the column-player, such that the amount of the column-player does not go below $x_k$.
Thus we have that the column-player indeed does not go below the amounts $\vec{x}$ and thus does not lose the game
if he makes moves in the described way.

This analysis shows all the statements of the lemma except for the last one. For this statement note that the column-player does not lose
in the position $r_j$ if he has a move to some position $c_i$ such that first, he does not lose immediately, and second,
he does not lose in position $c_i$. Thus $a_{ij}$ is finite and $x_i$ is finite and thus the $i$-th coordinate of $A \tp \vec{x}$ is finite. It is easy to reverse this argument.
\end{proof}

Next observation seems to be a new step towards min-plus linear duality.

\begin{lemma} \label{lem:winning}
There is a finite solution to $A \tp \vec{x} < B \tp \vec{x}$ iff the column-player has a winning strategy starting from any position.

There is a solution $\vec{x} \neq (\infty, \ldots, \infty)$ to $A \tp \vec{x} < B \tp \vec{x}$ iff the column-player has a winning strategy starting from some position.

There is a solution to $A \tp \vec{x} < B \tp \vec{x}$
with a finite coordinate $x_i$ iff the column-player has a winning strategy starting from the position $c_i$.

There is a solution to $A \tp \vec{x} < B \tp \vec{x}$
such that the $j$-th coordinate of $A \tp \vec{x}$ is finite iff the column-player has a winning strategy starting from the position $r_j$.

\end{lemma}
\begin{proof}
Suppose there is a solution $\vec{x}$ to $A \tp \vec{x} < B \tp \vec{x}$. Then for small enough positive $\eps$ there is a solution to
$A \tp \vec{x} \leq (B - \eps) \tp \vec{x}$, where we subtract $\eps$ from each entry of $B$.
Then by Lemma~\ref{lem:non-losing} there is a non-losing strategy for the column-player
in the mean payoff game $G^\p$ corresponding to the system $A \tp \vec{x} \leq (B - \eps) \tp \vec{x}$.
Let the column-player apply the same strategy to the game $G$ corresponding to $A \tp \vec{x} < B \tp \vec{x}$.
Then compared to the game $G^\p$ after $k$ moves the column-player will have at least the value $k\eps$
added to his amount. Since the amount of the column-player is bounded from below in $G^\p$ it will grow
to infinity in $G$. Thus in the game $G$ the column-player has a winning strategy.

For the opposite direction, assume that the column-player has a winning strategy. Then if we add a small enough $\eps$
to all payoffs of the row-player, the column-player will still have a winning strategy, which is in particular
non-losing. Thus we have by Lemma~\ref{lem:non-losing} that there is a solution $\vec{x}$ to $A \tp \vec{x} \leq (B  - \eps) \tp \vec{x}$,
where we subtract $\eps$ from each entry of $B$.
Clearly the very same $\vec{x}$ is a solution to $A \tp \vec{x} < B \tp \vec{x}$ and we are done.
\end{proof}

Now to get the Lemma~\ref{lem:min-plus_linear_duality} it is only left to use a duality of mean payoff games. For this note that for a game $G$ either the column-player has a winning strategy starting from some position, or the row-player has a non-losing strategy starting from the same position. Also note that if along with the game $G$ corresponding to the min-plus linear system $A \tp \vec{x} \leq B \tp \vec{x}$ we consider a game $G^\p$ corresponding to the min-plus linear system $B^T \tp \vec{x} \leq A^T \tp \vec{x}$, then the games $G$ and $G^\p$ are the same except the roles of the players switched.

\subsection{Tropical duality}

Suppose we are given a tropical linear system $A \tp \vec{x}$ for $A \in \bb{R}^{m\times n}$ and
we are interested whether it has a solution. First of all it is known that if the number of variables is greater than
the number of equations, then there is always a solution~\cite{DSS05rank}. So we can assume that $m \geq n$.
Next note that if we add the same number to all the entries in some row of $A$ then the set of solutions does not change.
One simple obstacle for $A \tp \vec{x}$ to have a solution is if we can add some numbers to all rows of $A$ and possibly
permute rows and columns in such a way that the minimums in the first $n$ rows of the resulting matrix are attained just
in entries $(1,1), (2,2), \ldots, (n,n)$. It is easy to see that if this is the case then there is no solutions to $A \tp \vec{x}$~\cite{DSS05rank}.
It turns out that this is the only obstacle. We give a proof below however we note that this is already implicit in~\cite{Grigoriev13complexity,IzhakianR2009rank}.

\begin{proof}[Proof of Lemma~\ref{lem:tropical_linear_duality}]
Given a tropical product of a matrix by a vector $A \tp \vec{a}$, where $A \in \bb{R}_{\infty}^{m \times n}$ it is convenient to introduce
a value $n_i(A \tp \vec{a})$ for all $i = 1, \ldots, m$ which is equal to the number of the column in which the finite minimum in the row
$i$ is situated (if there is one). If there are several minimums, $n_i(A \tp \vec{a})$ corresponds to the first one. When the matrix and the vector
are clear from the context we simply write $n_i$.

Denote by $C_i$ for $i=1, \ldots, n$ the matrix in $\bb{R}^{m\times n}$ with $1$ entries in the $i$-th
column and $0$ entries in other columns.
Denote by $R_i$ for $i=1, \ldots, n$ the matrix in $\bb{R}^{n\times m}$ with $1$ entries in the $i$-th
row and $0$ entries in other rows.
Note that $R_i = C_i^T$.

We will show the first part of the lemma. The proof of the second part is completely analogous.

Suppose we are given a matrix $A \in \bb{R}^{m\times n}_{\infty}$ and consider the tropical linear system $A \tp \vec{x}$.
As shown in~\cite{GP13Complexity} (cf. Section~\ref{sec:trop_vs_min-plus}) $\vec{x}$ is a solution to it iff for all small enough $\eps > 0$ $\vec{x}$ is a solution to the following min-plus system:
$$
\left(
  \begin{array}{c}
    A + \eps C_1  \\
    A + \eps C_2  \\
    \vdots \\
    A + \eps C_n\\
  \end{array}
\right) \tp \vec{x}
\leq
\left(
  \begin{array}{c}
    A  \\
    A  \\
    \vdots \\
    A\\
  \end{array}
\right) \tp \vec{x}.
$$

By Lemma~\ref{lem:min-plus_linear_duality} this system has a solution $\vec{x}$ with finite coordinates $x_i$ for $i \in S$ if and only if the system
\begin{align} \label{eq:tropical_transpose}
\begin{split}
&\left(
  \begin{array}{cccc}
    A^T  & A^T & \cdots & A^T\\
  \end{array}
\right) \tp \vec{y}
<\\
&\left(
  \begin{array}{cccc}
    A^T + \eps R_1  & A^T + \eps R_2  & \cdots &  A^T + \eps R_n\\
  \end{array}
\right) \tp \vec{y}
\end{split}
\end{align}
has no solution $\vec{y}$ such that for some $i \in S$  the $i$-th coordinate of
$$
\left(
  \begin{array}{cccc}
    A^T  & A^T & \cdots & A^T\\
  \end{array}
\right) \tp \vec{y}
$$
is finite.

On the right-hand side of~\eqref{eq:tropical_transpose} we have a block matrix with blocks $A^T + \eps R_i$.

It is left to show that the system~\eqref{eq:tropical_transpose} has a specified solution iff there is $\vec{z}$ such that in each row of $A^T \tp \vec{z}$ the minimum is attained at least once or is equal to $\infty$, for each two rows with the finite minimums these minimums are in different columns and
for some $i \in S$ the $i$-th coordinate of $A^T \tp \vec{z}$ is finite.
Note that if $\vec{y}$ is a solution to~\eqref{eq:tropical_transpose} then in each row $i$, where the minimum on the left-hand side is finite  we have $m(i-1) < n_i \leq mi$, that is $n_i$ is in the $i$-th block.
Indeed, otherwise the minimum in this row in the left-hand side is greater or equal than the minimum in the right-hand side,
since the $i$-th rows on the left-hand side and on the right-hand side differ only in the $i$-th block.
Thus if $\vec{y}$ is a solution then for each row $i$ with the finite minimum there exists a column $j_i = n_i \pmod m$ of the $i$-th block
such that the minimum is attained in this column. Note, that for $i_1 \neq i_2$ with finite minimums in the rows of the system we have $j_{i_1} \neq j_{i_2}$.
Otherwise rows $i_1, i_2$ and columns $n_{i_1} = m i_1 + j_{i_1}, n_{i_2} = m i_2 + j_{i_1}$ will form a $2 \times 2$ subsystem
\begin{align*}
&\left(
\begin{array}{cc}
    a_{i_1, j_{i_1}}  & a_{i_1, j_{i_1}} \\
    a_{i_2, j_{i_1}}  & a_{i_2, j_{i_1}}  \\
  \end{array}
\right) \tp
\left(
\begin{array}{cc}
    y_{m i_1 + j_{i_1}} \\
    y_{m i_2 + j_{i_1}} \\
  \end{array}
\right)
<\\
&\left(
\begin{array}{cc}
    a_{i_1, j_{i_1}} + \eps  & a_{i_1, j_{i_1}} \\
    a_{i_2, j_{i_1}}  & a_{i_2, j_{i_1}}  + \eps \\
  \end{array}
\right) \tp
\left(
\begin{array}{cc}
    y_{m i_1 + j_{i_1}} \\
    y_{m i_2 + j_{i_1}} \\
  \end{array}
\right),
\end{align*}
with finite $y_{m i_1 + j_{i_1}}, y_{m i_2 + j_{i_1}}$, which is impossible.
Thus columns $j_{i}$ correspond to different columns of the matrix $A^T$.
Let us consider the tropical system
$$
A^T \tp \vec{z}
$$
and consider the following vector $\vec{z}$. For all $i$ with the finite minimum in the row $i$ let $z_{j_i} = y_{mi + j_i}$. Set all other coordinates of $\vec{z}$ to $\infty$.
For this $\vec{z}$ the minimum in each row is either infinite or is attained once and no two minimums are in the same column.
Indeed, if a finite minimum is attained twice for some row, then
for the same row of~\eqref{eq:tropical_transpose} we will have equality.
Note also that the $i$-th coordinate of
$$
\left(
  \begin{array}{cccc}
    A^T  & A^T & \cdots & A^T\\
  \end{array}
\right) \tp \vec{y}
$$
is finite iff the $i$-th coordinate of $ A^T \tp \vec{z}$ is finite.

In the opposite direction, suppose we have a $\vec{z}$ such that in each row $A^T \tp \vec{z}$ the minimum is either infinite or is attained once
and no two minimums are in the same column. Then for each row $i$ with the finite minimum consider a column $j_i$ in which this minimum is attained and let $y_{im+j_i} = z_{j}$. Set all other coordinates of $\vec{y}$ to $\infty$.
Then for any small enough $\eps$ we will have a solution of~\eqref{eq:tropical_transpose} and the $i$-th coordinate of $A^T \tp \vec{z}$ is finite iff
the $i$-th coordinate of
$$
\left(
  \begin{array}{cccc}
    A^T  & A^T & \cdots & A^T\\
  \end{array}
\right) \tp \vec{y}
$$
is finite.

\end{proof}

{\bf Acknowledgements}. The first author is grateful to Max-Planck Institut f\"ur
Mathematik, Bonn for its hospitality during the work on this paper.

Part of the work of the second author was done during the visit to Max-Planck Institut f\"ur
Mathematik, Bonn.

{\small
\bibliographystyle{abbrv}
\bibliography{bib/null}
}

\end{document}